\documentclass[a4paper,11pt]{amsart}
\usepackage{amsmath,amssymb,amsthm,bbm,mathtools}
\usepackage[utf8x]{inputenc}
\usepackage{enumitem}
\usepackage[font={small}]{caption}
\usepackage[mathscr]{euscript}
\usepackage{hyperref}
\usepackage[left=30mm,right=30mm,top=25mm,bottom=25mm]{geometry}

\usepackage{tikz}
\usetikzlibrary{shapes,arrows,calc,positioning,fit}
\pgfdeclarelayer{background}
\pgfsetlayers{background,main}

\newcommand\parit[1]{\textup{(\textit{#1})}}

\newcommand\parStyle[1]{\textrm{\mdseries\upshape({#1}\kern0.1ex)}}
\newenvironment{romanum}{%
	\begin{enumerate}[label=\parStyle{\itshape\roman*},labelwidth=\romanumlabelwd,leftmargin=1\romanumlabelwd,itemsep=0.2ex]%
}{%
	\end{enumerate}%
}
\newlength\romanumlabelwd
 \newtheorem{theorem}{Theorem}[section]
 \newtheorem{lemma}[theorem]{Lemma}
 \newtheorem{proposition}[theorem]{Proposition}
 \newtheorem{corollary}[theorem]{Corollary}

 \newtheorem{definition}[theorem]{Definition}
 \newtheorem{example}[theorem]{Example}
  \newtheorem{remark}[theorem]{Remark}

\DeclareMathOperator\Span{span}

\DeclareMathOperator\Tr{Tr}

\DeclareMathOperator\cb{cb}
\DeclareMathOperator\cbdiam{cb-diam}

\let\Re\RealPart
\let\Im\ImagPart
\newcommand\C{\mathbb{C}}

\renewcommand\vec\mathbf

\newcommand\ox{\otimes}


\let\subset\subseteq
\let\oldepsilon\epsilon
\let\epsilon\varepsilon
\let\varepsilon\oldepsilon
\let\le\leqslant
\let\ge\geqslant

\DeclareMathOperator\id{id}
\DeclareMathOperator\diam{diam}
\DeclareMathOperator\sdiam{sdiam}
\DeclareMathOperator\sigmadiam{\sigma\mathrm{diam}}

\DeclareMathOperator\cbsdiam{cb-sdiam}

\def\href#1#2{\texttt{#2}}% For when hyperref is not defined

\makeatletter
	\newcommand\ket[1]{\left| #1 \right\rangle\@ifnextchar\bra{\mspace{-4mu}}{}}
	\newcommand\bra[1]{\left\langle #1 \right|}

\makeatother

\begin{document}

%\begin{center}
%{\LARGE On spectral diameters and linear superoperators}
%\vskip 12 pt
%{\Large \{Christopher, Niel\}}
%\vskip 12 pt
%{\Large 11 October 2019}
%\end{center}

\title{On numerical diameters and linear maps}

\author{Niel de Beaudrap}
\address{School of Engineering and Informatics, University of Sussex, \newline
\hspace*{\parindent}  Brighton, UK}
\email{niel.debeaudrap@sussex.ac.uk}

\author{Christopher Ramsey}
\address{Department of Mathematics and Statistics, MacEwan University, \newline
\hspace*{\parindent}  Edmonton, Alberta, Canada}
\email{ramseyc5@macewan.ca}

%\date{\today}

\makeatletter
\@namedef{subjclassname@2020}{%
  \textup{2020} Mathematics Subject Classification}
\makeatother


\begin{abstract}
This paper studies the diameter of the numerical range of bounded operators on Hilbert space and the induced seminorm, called the numerical diameter, on bounded linear maps between operator systems which is sensible in the case of unital maps and their scalar multiples. It is shown that the completely bounded numerical diameter is a norm that is comparable but not equal to the completely bounded norm. This norm is particularly interesting in the case of unital completely positive maps and their sections. 
\end{abstract}

\maketitle

\section{Introduction}

The numerical range and its maximum distance from the origin, the numerical radius, are widely studied mathematical objects, originating with Hausdorff and Toeplitz, see \cite{BonsallDuncan} for a thorough treatment of the subject. Numerical ranges have also been widely studied in quantum information theory \cite{Cao, Kribsetal} and in other applied settings as the numerical range can be a reasonable approximation for the spectrum.  

This paper concerns the related concept of the numerical diameter which is the diameter of the numerical range. This is especially well-behaved in the context of normal operators as the numerical diameter is equal to the spectral diameter, the diameter of the spectrum. 
This concept also has motivations from quantum information theory, where for a bounded Hermitian operator $H$, it represents the minimal norm of an operator $H + c1 \ge 0$, essentially taking the minimum eigenvalue of $H$ as a reference point for all other eigenvalues.
However, the numerical diameter has been studied far less than the numerical radius, appearing only a few times in the literature over the years \cite{Grabiner, Parlett, CNM, BourinLee} with different results from what we are presenting here. 

Section 2 introduces the spectral and numerical diameters of an operator and establishes that the numerical diameter is a seminorm that is invariant under translations of scalar multiples of the identity. In particular, the relationship to the operator norm and the numerical radius are worked out. 

Section 3 introduces the main consideration of the paper, which is the induced seminorm (again called the numerical diameter) of linear maps between operator systems given by the numerical diameter.
We describe an application for numerical diameters, on operators and on linear maps, to quantum information theory.
Recall, that  a (concrete) operator system is a self-adjoint, unital, closed subspace of bounded operators on a Hilbert space, $\mathcal S\subseteq B(\mathcal H)$ . That is, $E^*\in\mathcal S$ for every $E\in \mathcal S$. The self-adjoint operators will be denoted $\mathcal S_{sa}$. 
In this context, the bounded linear maps with a bounded numerical diameter are precisely those which are scalar multiples of unital maps, here called paraunital maps. As is typical with linear maps of operators much can be learned by considering their complete structure. To this end, we introduce the completely bounded numerical diameter. This turns out to be a norm comparable to, but distinct from, the cb-norm for paraunital maps, Theorem \ref{thm:cbdiamnorm}.

Section 4 studies unital completely positive maps and their sections (that is, their right inverses), and their behaviour with respect to the numerical diameter. It is shown that a ucp map has a contractive cb-numerical diameter, Proposition \ref{prop:contractivediameter}. This is used along with the rest of the section to prove that a section of a ucp map is an isometry if and only if the cb-numerical diameter is equal to 1, Theorem \ref{thm:cbdiamisometry}. The rest of the section explores the best that can be said of approximately isometric maps.

Section 5 considers translating finite-dimensional self-adjoint maps by multiples of the trace. There is always such a translation that makes the map completely positive, Theorem \ref{thm:tracecp}. One needs additional conditions, namely scaled trace-preserving, to ensure that there is a translation that is a section of a completely positive map. These translations fit in with the rest of the paper as they do not change the numerical diameter of the map. However, translations by scalar multiples of the trace usually change the completely bounded diameter so it does come at a cost.

\section{Spectral and Numerical Diameters of Operators}

We begin by introducing the main concepts of this paper.

\begin{definition}
For a (possibly infinite-dimensional) Hilbert space $\mathcal H$ and for $E \in B(\mathcal H)$, the \textbf{spectral diameter} of $E$ is  
\[
\lVert E \rVert_{\sigmadiam} = \sup\{|\lambda - \mu| : \lambda,\mu\in \sigma(E)\}
\]
where $\sigma(E)$ is the spectrum of eigenvalues of $E$; and the \textbf{spectral radius} of $E$ is
\[
\rho(E) = \sup\{|\lambda| : \lambda\in\sigma(E)\}\, .
\]
In the case where $E$ is self-adjoint this is more simply given as
\[
\lVert E \rVert_{\sigmadiam} = \lambda_{\max}(E) - \lambda_{\min}(E) \quad \textrm{and} \quad \rho(E) = \max\;\bigl\{ |\lambda_{\max}(E)|,\, |\lambda_{\min}(E)| \bigr\}\, .
\] 
\end{definition}

The spectrum of an operator is of course compact, so these suprema are actually maxima.
These two quantities impart some useful information for self-adjoint operators, and more generally for normal operators.
Specifically, for the spectral radius we have Gelfand's formula:
\[
\rho(E) = \lim_{n\rightarrow \infty} \|E^n\|^{1/n}\, .
\]
By functional calculus this implies that for $E\in B(\mathcal H)$ normal 
\[
\rho(E) = \|E\|,
\]
a result which one can also obtain from the Spectral Theorem.
However, the spectral diameter and spectral radius become trivial for many non-normal operators, e.g.,~anything with $\sigma(E) = \{0\}$. 
One can sensibly extend these definitions to non-normal operators through the numerical range.

\begin{definition}
For a Hilbert space $\mathcal H$ and an operator $E\in B(\mathcal H)$ the \textbf{numerical range} of $E$ is 
\[
W(E) = \left\{\frac{\langle Ev, v\rangle}{\langle v,v\rangle} : v\in \mathcal H, v\neq 0\right\}.
\]
\end{definition}

The numerical range is a convex set (the Toeplitz-Hausdorff theorem, \cite{Toeplitz, Hausdorff}) and is easily seen to be bounded, living in the disk of radius $\|E\|$. 
(That is, we have $W(E) \subseteq \| E \| \,\overline{\mathbb D}$, where $\mathbb D = \{ z \in \mathbb C : \lvert z \rvert < 1\}$, and $\overline{\mathbb D}$ is its closure.
It will occasionally be useful to describe applications of the numerical range to self-adjoint bounded operators which are related to a given bounded operator $E$:
\begin{lemma}
If $E\in B(\mathcal H)$, then
\[
W(\Re(E)) = \{ \Re(\lambda)  :  \lambda \in W(E)\} \quad \textrm{and} \quad W(\Im(E)) = \{ \Im(\lambda)  :  \lambda\in W(E)\}\,.
\]
\end{lemma}
\begin{proof}
Let $v\in \mathcal H$ with $\|v\|=1$. Then
\begin{align*}
    \Re\left(\langle Ev,v\rangle\right) \;&=\; \tfrac{1}{2}\left(\langle Ev,v\rangle + \overline{\langle Ev,v\rangle}\right)
    \;=\; \tfrac{1}{2}\bigl(\langle Ev,v\rangle + \langle E^*v,v\rangle\bigr)
     \;=\; \langle \Re(E)v,v\rangle\,.
\end{align*}
The imaginary calculation follows similarly.
\end{proof}

Recall that the states over a C$^*$-algebra are the positive linear functionals with norm~1.
(This coincides in finite dimensions with maps of the form $E \mapsto \Tr(E q)$, for an arbitrary positive-semidefinite operator $q$ with unit trace.) In particular, states are related to the numerical range since $E \mapsto \langle E v,v\rangle$ for any $v\in \mathcal H, \|v\|=1$ is a state.

\begin{proposition}
For $E\in B(\mathcal H)$
\begin{align*}
\overline{W(E)} \ & = \ \left\{ \phi(E) : \phi \ \textrm{is a state on} \ B(\mathcal H)  \right\}
\\ & = \ \left\{ \phi(E) : \phi \ \textrm{is a state on} \ C^*(E)\right\}
\\ & = \ \left\{ \phi(E) : \phi \ \textrm{is a state on} \ \mathcal S\right\}
\end{align*}
for any operator system $\mathcal S \subseteq B(\mathcal H)$ such that $E\in \mathcal S$.
\end{proposition}
\begin{proof}
The first result can be found recorded in \cite{Arveson} but is certainly older. Note that the second equality follows from the third.

Now, every restriction of a state on $B(\mathcal H)$ to $\mathcal S$ is clearly a state on $\mathcal S$.
Conversely, by Arveson's Extension Theorem, cf. \cite[Theorem 7.5]{Paulsenbook}, every state on $\mathcal S$ can be extended to a state on $B(\mathcal H)$. 
\end{proof}

Identifying the closure of the numerical range with the image of the state space allows us to be vague about the concrete context of our operator systems. In particular, the closure of the numerical range is preserved by completely isometric order isomorphisms of the operator systems. One should note that $\overline{W(E)}$ is also equal to the images under the unital completely contractive linear functionals on a unital operator space $\mathcal M \subseteq B(\mathcal H)$ with $E\in \mathcal M$. Many of these ideas are pursued further in \cite{Farenick}.

\begin{proposition}[Theorem 8.14 \cite{Stone}]\label{prop:stone}
For every normal $E\in B(\mathcal H)$
\[
\overline{W(E)} \ = \ \operatorname{conv}(\sigma(E)),
\]
the convex hull of the spectrum.
\end{proposition}

This highlights that the closure of the numerical range is a better object of study than the spectrum since it allows us to be context-free.

\begin{example}
As an operator system $C(\mathbb T)$ embeds completely isometrically into $C(\overline{\mathbb D})$ by the identity map $z^n\mapsto z^n$ for all $n\in\mathbb Z$, just not multiplicatively. Then $\sigma_{C(\mathbb T)}(z) = \mathbb T$ and $\sigma_{C(\overline{\mathbb D})}(z) = \overline{\mathbb D}$, which of course both have the same convex hull.
\end{example}

The upshot of the previous proposition is that if $E \in \mathcal S$ is a self-adjoint element in an operator system then $\lambda_{\max}(E)$ and $\lambda_{\min}(E)$ are well-defined regardless of how $\mathcal S$ is embedded into a specific $B(\mathcal H)$.

One can then extend the definitions of the spectral diameter and radius to something better behaved on non-normal operators, as follows.

\begin{definition}
For a Hilbert space $\mathcal H$ and an operator $E\in B(\mathcal H)$ the \textbf{numerical diameter} of $E$ is 
\[
\|E\|_{\mathrm{diam}} = \sup \, \bigl\{ |\lambda - \mu| : \lambda,\mu \in W(E) \bigr\}
\]
and the \textbf{numerical radius} of $E$ is
\[
r(E) = \sup\{|\lambda| : \lambda \in W(E)\}\, .
\]
\end{definition}

By Proposition \ref{prop:stone} for any normal $E\in B(\mathcal H)$ we have
\[
\|E\|_{\sigmadiam} = \|E\|_{\diam} \quad \textrm{and} \quad \rho(E) = r(E).
\]

\begin{remark}
The width of the numerical range has also been considered, that is, the smallest distance between parallel lines that touch the boundary of the numerical range \cite{BourinMhanna, CNM}. As well, \cite{BourinLee2} studies the inner diameter of the numerical range, the diameter of the largest disk that the numerical range contains.
\end{remark}

%The following simple observation serves to motivate using the numerical diameter in place of the spectral diameter in the case of self-adjoint operators:
%\begin{proposition}
%    Let $E\in B(\mathcal H)$. 
%    If $E$ is self-adjoint, then $\lVert E \rVert_{\sigmadiam} = \lVert E \rVert_{\diam}$.
%\end{proposition}

We now prove many elementary properties of the numerical diameter.

\begin{proposition}\label{prop:diamcontinuous}
For any $E \in B(\mathcal H)$, $\lVert E \rVert_{\mathrm{diam}} \le 2 \lVert E \rVert$ and $\| E\|_{\diam} \leqslant 2r(E)$. 
\end{proposition}
\begin{proof}
As we noted above, for every $v\in \mathcal H, v \neq 0,$ one has
\[
  \left|\frac{\langle Ev, v\rangle}{\langle v,v\rangle}\right|
\;\leqslant\;
  \|E\|\left|\frac{\langle v, v\rangle}{\langle v,v\rangle}\right|
\;=\;
  \|E\|.
\]
Thus, 
\begin{align*}
    \|E\|_{\mathrm{diam}}
  \;& =\;
    \sup\; \bigl\{ |\lambda - \mu| : \lambda,\mu \in W(E) \bigr\}
\\[.5ex] & = \;
  \sup\left\{ \left|\frac{\langle Ev, v\rangle}{\langle v,v\rangle} - \frac{\langle Ew, w\rangle}{\langle w,w\rangle}\right| : v,w\in\mathcal H, v,w\neq0 \right\}
%\\[.5ex] &
  \;\leqslant\;
  2\|E\|.
\end{align*}
The second inequality follows more simply from the triangle inequality.
\end{proof}

\begin{proposition}[\cite{BonsallDuncan}]
The numerical radius is a norm on $B(\mathcal H)$ with
\[
r(E) \leqslant \|E\| \leqslant 2r(E) \quad \textrm{and} \quad r(E^*) = r(E)
\]
for every $E\in B(\mathcal H)$.
\end{proposition}

\begin{proposition}%
\label{prop:diameter-real-projection}%
Let $E\in B(\mathcal H)$. 
Then 
\[
    \mathclap{\phantom{\max\limits_c}} \lVert E \rVert_{\diam} \;=\;
    \max\limits_{c}\ \bigl\lVert \Re(cE) \bigr\rVert_{\diam} \;=\;
    \max\limits_{c}\ \bigl\lVert \Im(cE)\bigr\rVert_{\diam} \;=\;
    \max\limits_{a,b}\ \tfrac{1}{2} \bigl\lVert a E + bE^* \bigr\rVert_{\diam}
\]
where the maxima are taken over $a,b,c \in \mathbb C$ such that $\lvert a \rvert = \lvert b \rvert = \lvert c \rvert = 1$.
\end{proposition}
\begin{proof}
    For an arbitrary $c \in \mathbb C$ with $\lvert c \rvert = 1$, we clearly have $\Re(cE) = \tfrac{1}{2}(cE + c^*E^*) = \tfrac{1}{2}c(E + c^{-2} E^*)$.
    Letting $s = c^{-2}$, we then have
    \[
        \bigl\lVert \Re(cE) \bigr\rVert_{\diam} 
        \,=\;
        \tfrac{1}{2}\bigl\lVert cE + c^*E^* \bigr\rVert_{\diam}
        \,=\;
        \tfrac{1}{2}\bigl\lVert c(E + sE^*) \bigr\rVert_{\diam}
        \,=\;
        \tfrac{1}{2}\bigl\lVert E + sE^* \bigr\rVert_{\diam}\;. 
    \]
    For an arbitrary $a \in \mathbb C$ with $\lvert a \rvert = 1$, let $b = as$.
    We then also have
    \[
        \tfrac{1}{2}\bigl\lVert E + sE^* \bigr\rVert_{\diam}
        \,=\;
        \tfrac{1}{2}\bigl\lVert a^*(aE + bE^*) \bigr\rVert_{\diam}
        \,=\;
        \tfrac{1}{2}\bigl\lVert aE + bE^* \bigr\rVert_{\diam}    \;.
    \]
    Conversely, any such choices of $a,b,s \in \mathbb C$ determine a value of $c \in \mathbb C$ up to a factor of $\pm 1$.
    As $\bigl\lVert \Re(-cE) \bigr\rVert_{\diam} = \bigl\lVert \Re(cE) \bigr\rVert_{\diam}$\,, the range of the expressions $\bigl\lVert \Re(cE) \bigr\rVert_{\diam}$ and of $\tfrac{1}{2}\bigl\lVert aE + bE^* \bigr\rVert_{\diam}$
    are the same,
    as functions of complex units $a,b,c \in \mathbb C$.
    Taking $c' = -ic$, we obtain similar results for $\Im(cE) = -\tfrac{1}{2}i( cE  - c^*\!\!\;E^*) = \Re(c'E)$.    
    
    It then suffices to show that $\lVert E \rVert_{\diam}$ is the maximum value of these quantities over $a,b,c \in \mathbb C$ with $\lvert a \rvert = \lvert b \rvert = \lvert c \rvert = 1$.
    We may show that $\lVert E \rVert_{\diam}$ is indeed an upper bound:
    \[
        \tfrac{1}{2}\bigl\lVert aE + bE^* \bigr\rVert_{\diam}
        \:\!\le\,
        \tfrac{1}{2}\Bigl(\bigl\lVert aE  \bigr\rVert_{\diam} + \bigl\lVert bE^* \bigr\rVert_{\diam}\Bigr)
        =\,
        \tfrac{1}{2}\Bigl(\lVert E \rVert_{\diam} + \lVert E^* \rVert_{\diam}\Bigr)
        \:\!=\,
        \lVert E \rVert_{\diam} \,.
    \]
    To show that this bound can be achieved, let $\lambda, \mu \in \overline{W(E)}$ be such that $\lvert \lambda - \mu \rvert = \lVert E \rVert_{\diam}$.
    If $\lambda = \mu$, it follows that $E \propto 1_{B(\mathcal H)}$\,, in which case the proposition holds.
    Otherwise, consider
    \[
        c \;=\; \frac{\lambda^* - \mu^*}{\lvert \lambda - \mu \rvert},
    \]
    so that $c\lambda - c\mu = \lvert \lambda - \mu \rvert$.
    Then $c\lambda = p + ir$ and $c\mu = q + ir$ for some real values $p,q,r \in \mathbb R$, where furthermore $p > q$.    Hence, by taking closures in the previous lemma, $p = \Re(c\lambda)$ and $q = \Re(c\mu)$ are both elements of $\overline{W\bigl(\Re(cE)\bigr)}$, from which it follows that ${\bigl\lVert \Re(cE) \bigr\rVert_{\diam} \ge p - q = \lVert E \rVert_{\diam}}$\,.
\end{proof}

\begin{proposition}\label{prop:diamseminorm}
  The numerical diameter is a seminorm on $B(\mathcal H)$, but not a norm, with $\|E^*\|_{\diam} = \|E\|_{\diam}$ for every $E\in B(\mathcal H)$. Moreover, $\|E\|_{\mathrm{diam}} = 0$ if and only if $E=c1_\mathcal H$ for some $c\in \mathbb C$. 
\end{proposition}
\begin{proof}
  The numerical diameter is not positive definite since, $\lVert c1_{\mathcal H} \rVert_{\mathrm{diam}} = 0$.
  However, it is evidently positive semidefinite and homogeneous, and it satisfies the triangle inequality:
\[
\begin{aligned}[b]{}
    \mspace{-18mu}
    \lVert E + F \rVert_{\mathrm{diam}}
  &=\;
    \sup \,\Biggl\{ \left|\frac{\langle (E+F)v, v\rangle}{\langle v,v\rangle} - \frac{\langle (E+F)w, w\rangle}{\langle w,w\rangle}\right| : v,w\in\mathcal H, v,w\neq0\Biggr\}
    \mspace{-30mu}
  \\[1ex]&\leqslant\;
    \begin{aligned}[t]
        \sup\; \Biggl\{ \biggl|
      &
        \frac{\langle Ev, v\rangle}{\langle v,v\rangle}
        - \frac{\langle Ew, w\rangle}{\langle w,w\rangle}\biggr| 
      \\&  + \biggl|\frac{\langle Fv, v\rangle}{\langle v,v\rangle}
        - \frac{\langle Fw, w\rangle}{\langle w,w\rangle}\biggr| : v,w\in\mathcal H, v,w\neq0\Biggr\}    
      \end{aligned}
  \\&\leqslant\;
    \lVert E \rVert_{\mathrm{diam}} + \lVert F \rVert_{\mathrm{diam}} \,.   
\end{aligned}
\]
The adjoint preserves the numerical identity quite simply since
\begin{align*}
\|E^*\|_{\diam} & = \sup \,\Biggl\{ \left|\frac{\langle E^*v, v\rangle}{\langle v,v\rangle} - \frac{\langle E^*w, w\rangle}{\langle w,w\rangle}\right| : v,w\in\mathcal H, v,w\neq0\Biggr\}
\\ & = \sup \,\Biggl\{ \left|\frac{\langle v, Ev\rangle}{\langle v,v\rangle} - \frac{\langle w, Ew\rangle}{\langle w,w\rangle}\right| : v,w\in\mathcal H, v,w\neq0\Biggr\}
\\ & = \sup \,\Biggl\{ \left|\overline{\frac{\langle Ev, v\rangle}{\langle v,v\rangle} - \frac{\langle Ew, w\rangle}{\langle w,w\rangle}}\right| : v,w\in\mathcal H, v,w\neq0\Biggr\}
\\ & = \|E\|_{\diam}\,.
\end{align*}

Lastly, suppose $\|E\|_{\mathrm{diam}} = 0$. This implies that the numerical range is a single point, the only convex set with diameter 0. Hence, say $W(E) = \{c\}$ for $c\in \mathbb C$. This implies that $W(E - c1_\mathcal H) = \{0\}$ and so
\[
\langle (E-c1_\mathcal H)v,v\rangle = 0, \ \forall v\in\mathcal H.
\]
Therefore, $E = c1_\mathcal H$.
\end{proof}

\begin{corollary}
If $E\in B(\mathcal H)$ and $c\in \mathbb C$ then
\[
\|E + c1_{\mathcal H}\|_{\diam} = \|E\|_{\diam}.
\]
\end{corollary}

We can use this translation invariance to get some lower bounds of the numerical diameter.

\begin{proposition}\label{prop:diameterlowerbounds}
Let $E\in B(\mathcal H)$. 
\begin{enumerate}
\item 
    There exists a $c_E\in \mathbb C$ such that $
r(E-c_E 1_{\mathcal H}) \leqslant \frac{1}{\sqrt 3}\|E\|_{\diam}$.
\item
    For $E$ self-adjoint, there exists a $k_E \in \mathbb C$ such that $\|E-k_E 1_{\mathcal H}\| = \frac{1}{2}\|E\|_{\diam}$.
    \smallskip
\item
    For any $E$, $\bigl\lVert E - (k_{\!\:\Re(E)} + i\, k_{\!\:\Im(E)})1_{\mathcal H} \bigr\rVert \leqslant \|E\|_{\diam}$.
\end{enumerate}
\end{proposition}
\begin{proof}
For the first inequality, Jung's Theorem \cite{Jung} gives that every bounded convex set in the complex plane of diameter $d$ lies within (and on the boundary of) a circle of radius $\frac{1}{\sqrt 3}d$. This is an optimal bound as the equilateral triangle of side length 1 has diameter 1 and circumradius $\smash{\frac{1}{\sqrt 3}}$. In our case, choose a circle of radius $\frac{1}{\sqrt 3}\|E\|_{\diam}$ centred at $c_E$ that encircles $W(E)$. Thus, the operator $E-c_E 1_\mathcal H$ has numerical range inside a circle of radius $\frac{1}{\sqrt 3}\|E\|_{\diam}$ centred at the origin. This then gives the desired bound, which is saturated for the diagonal unitary matrix $E = 1 \oplus \omega \oplus \omega^2 \in M_3(\mathbb C)$ for $\omega = \exp(2 \pi i /3)$.

For the second result, where $E$ is self-adjoint, define
\[
k_E \ =\  \tfrac{1}{2}\bigl(\lambda_{\max}(E) + \lambda_{\min}(E)\bigr) \ \in \ \mathbb R.
\]
This implies that,
\[
\lambda_{\max}(E - k_E 1_{\mathcal H}) \ =\  \lambda_{\max}(E) - k_E \ =\  \tfrac{1}{2}\bigl(\lambda_{\max}(E) - \lambda_{\min}(E)\bigr)
\]
and
\[
\lambda_{\min}(E - k_E 1_{\mathcal H})\  =\  \lambda_{\min}(E) - k_E \ =\  -\tfrac{1}{2}\bigl(\lambda_{\max}(E) - \lambda_{\min}(E)\bigr)
\]
Hence,
\begin{align*}
\|E-k_E 1_{\mathcal H}\| \ & = \ \max \Bigl\{ \bigl|\lambda_{\max}(E - k_E 1_{\mathcal H})\bigr|\;,\ \bigl|\lambda_{\min}(E - k_E 1_{\mathcal H})\bigr| \Bigr\}
\\ & = \ \tfrac{1}{2}(\lambda_{\max}(E) - \lambda_{\min}(E))
% \\ & 
\ = \ \tfrac{1}{2}\|E\|_{\diam}.
\end{align*}
Finally, using Proposition~\ref{prop:diameter-real-projection},
for any $E\in B(\mathcal H)$ we have
\begin{align*}
    \bigl\lVert E - (k_{\!\:\Re(E)} + i\, k_{\!\:\Im(E)})1_{\mathcal H} \bigr\rVert
    \ & \leqslant \ 
    \bigl\lVert\!\:\Re(E) - k_{\!\:\Re (E)} 1_{\mathcal H} \:\! \bigr\rVert
    + 
    \bigl\lVert\!\:\Im(E) - k_{\!\:\Im (E)}1_\mathcal H \:\!\bigr\rVert
\\[1ex] & = \ 
    \tfrac{1}{2}\bigl\lVert\!\:\Re(E)\bigr\rVert_{\diam} +\; \tfrac{1}{2}\bigl\lVert\!\:\Im(E)\bigr\rVert_{\diam}
\\[1ex] & \le \ 
    \tfrac{1}{2}\lVert E\rVert_{\diam} +\; \tfrac{1}{2}\lVert E\rVert_{\diam}
\ = \ \|E\|_{\diam}\,.\qedhere
\end{align*}
\end{proof}

We prove one final lemma in this section giving a taste of what numerical diameter calculations are like.
\begin{lemma}%
    \label{lemma:2x2-matrix-unit-numerical-diam}%
    Let $E \in B(\mathcal H)$ be given by $E = u v^\ast$, for unit vectors $u, v \in \mathcal H$.
    Then $\lVert E \rVert_{\diam} = 1$.
\end{lemma}
\begin{proof}
    Both $E$ and $E^*$ are 0 when restricted to $\operatorname{span}\{u,v\}^\perp$ and so without loss of generality we can just consider the case $\mathcal H = \mathbb C^2$.
    
    If $cu=v$, the result is evident, as then $E$ is proportional by a unit scalar to a rank-1 self-adjoint projection.
    Otherwise, let 
    $c \in \mathbb C$ with $\lvert c \rvert = 1$, such that $s = c v^* u \in [0,1)$.
    Let $w = \bar c v$, so that $w^* u = s$: then $F = c E =  u  w^*$  has the same numerical diameter as $E$.
    Let $ a = {( u +  w)/\lVert  u + w \rVert}$ and $b = {(u -  w)/\lVert  u -  w\rVert}$.
    Then $\{a, b\}$ is an orthonormal basis, and
    \[
     u = \sqrt{\tfrac{1}{2}(1 + s)\:\!} \, a + \sqrt{\tfrac{1}{2}(1 - s)\:\!} \, b,
    \qquad
     w = \sqrt{\tfrac{1}{2}(1 + s)\!\:}\,  a - \sqrt{\tfrac{1}{2}(1 - s)\:\!} \, b.
    \]
    We may then express $F$ as
    \[
        F = \tfrac{1}{2}(1+s)  a  a^* + \sqrt{\tfrac{1}{4}(1-s^2)} \bigl[ b  a^* -  a  b^* \bigr]
        - \tfrac{1}{2}(1-s)  b  b^*.
    \]
    Consider an arbitrary unit vector $ x = x_1 a + x_2  b \in \mathbb C^2$.
    We then have
    \[
         x^* F  x = {\tfrac{1}{2}(1+s)\:\!\lvert x_1 \rvert^2} - {\tfrac{1}{2}(1-s)\:\!\lvert x_2 \rvert^2}
    \]
    which is a convex combination of $\tfrac{1}{2}(1+s)$ and $-\tfrac{1}{2}(1-s)$, and attains these two extrema for $ x =  a$ and $ x =  b$ respectively.
    Thus $\lVert E \rVert_{\diam} = \lVert F \rVert_{\diam} = 1$.
\end{proof}

%While one can certainly proceed with the more general numerical diameter, the focus of this paper is really on Hermitian operators and self-adjoint maps on finite dimensional Hilbert space.
%Below, we restrict ourselves to this context.

\section{Numerical Diameters of Linear Maps}

% We now turn to linear maps and the spectral/numerical diameter.

One can apply the notion of numerical diameter to the problem of operationally distinguishing quantum states on a given system.
Consider a UCP map $\Psi$, taken as an action on quantum observables.
We may consider an operator $E$ with spectrum $\sigma(E) = \{\pm 1\}$ to represent an observable, which plays a role in a procedure to distinguish some pair of states $q$ and $q'$, essentially playing the role of a linear classifier, with the spectral diameter representing some notion of maximum capacity to distinguish pairs of states.
In this case, $\Psi(E)$ represents an observable with a potentially degraded distinguishing capacity, precisely according to how much smaller the spectral diameter of $\Psi(E)$ is from $E$ itself.
Considering a unital section $\Phi$ of $\Psi$, the degree to which $\Psi$ may diminish the distinguishing capabaility of such observables, is in inverse proportion to how much $\Phi$ may dilate the spectral diameters of its arguments.
This serves as a practical motivation to consider the induced seminorms.

\begin{definition}
Let $\Phi: \mathcal S \to \mathcal S'$ be a linear map between operator systems where $\mathcal S\neq {\mathbb C\cdot 1_\mathcal S}$. The \textbf{numerical diameter} $\lVert \Phi \rVert_{\mathrm{diam}} \in [0,\infty]$ is
\[
  \lVert \Phi \rVert_{\mathrm{diam}}
=
  \sup \,\left\{ \,
    \frac{\lVert \Phi(E) \rVert_{\mathrm{diam}}}{\lVert E \rVert_{\mathrm{diam}}} : 
    E \in \mathcal S \text{ and } E \notin  \mathbb C\cdot 1_{\mathcal S}\right\}
\]
Similarly, the \textbf{self-adjoint numerical diameter} $\|\Phi\|_{\sdiam} \in[0,\infty]$ is
\[
  \lVert \Phi \rVert_{\mathrm{sdiam}}
=
  \sup \,\left\{ \,
    \frac{\lVert \Phi(E) \rVert_{\mathrm{diam}}}{\lVert E \rVert_{\mathrm{diam}}} : 
    E \in \mathcal S_{sa} \text{ and } E \notin  \mathbb C\cdot 1_{\mathcal S}\right\}
\]
\end{definition}

\smallskip
\noindent
In the above definition of $\lVert \Phi \rVert_{\sdiam}$, the restriction of the supremum to $E \in \mathcal S_{sa}$ is significant only if $\Phi$ is not self-adjoint:

\begin{proposition}
    \label{prop:induced-sa-diam-on-sa-maps}%
    For $\Phi: \mathcal S \to \mathcal S'$ a self adjoint map between operator systems, $\lVert \Phi \rVert_{\diam} = \lVert \Phi \rVert_{\sdiam}$\,.
\end{proposition}
\begin{proof}
    % If $\lVert \Phi \rVert_{\diam} = 0$, then the result follows from Proposition~\ref{prop:seminorm}.
    % Otherwise,
    Let $x \in \mathcal S$ with $\lVert x \rVert_{\diam} \neq 0$, noting that if there is no such $x$ then $\mathcal S = \mathbb C 1_\mathcal S$ and the result is immediate. 
    % Otherwise, let $x \in B(\mathcal H)$ with $\lVert x \rVert \neq 0$.
    Now for any $s \in \mathbb C$, we have
    \[
        \Phi\bigl(\Re(sx)\bigr)
        \;=\;
        \Phi\bigl(\tfrac{1}{2}sx + \tfrac{1}{2}s^*x^*\bigr)
        \;=\;
        \tfrac{1}{2}s\Phi(x) + \tfrac{1}{2}s^*\Phi(x)^*
        \;=\;
        \Re\bigl(s\Phi(x)\bigr) \;.
    \]
    Using Proposition~\ref{prop:diameter-real-projection}, there is then $c \in \mathbb C$ with $\lvert c \rvert = 1$ such that $\lVert x \rVert_{\diam} = \lVert \Re(cx)\rVert_{\diam}$, which in particular maximises the value of $\lVert \Re(sx) \rVert_{\diam}$ for $\lvert s \rvert = 1$; and there is $c' \in \mathbb C$ with $\lvert c' \rvert = 1$ such that
    \[
    \begin{aligned}[b]
        \bigl\lVert \Phi(x) \bigr\rVert_{\diam}
    \;=\;
        \bigl\lVert \Re\bigl(c'\Phi(x)\bigr) \bigr\rVert_{\diam}
    \;&=\;
        \bigl\lVert \Phi\bigl(\Re(c'x)\bigr) \bigr\rVert_{\diam}
    \\&\le\;
        \lVert \Phi \rVert_{\sdiam}
        \cdot
        \bigl\lVert \Re(c'x) \bigr\rVert_{\diam}
    \\&\le\;
        \lVert \Phi \rVert_{\sdiam}
        \cdot
        \bigl\lVert \Re(cx) \bigr\rVert_{\diam}
    \;=\;
        \lVert \Phi \rVert_{\sdiam}
        \cdot
        \lVert x \rVert_{\diam} \;.
    \end{aligned}
    \]    
    Thus $\lVert \Phi \rVert_{\diam} \le \lVert \Phi \rVert_{\sdiam}$ for such a map $\Phi$, with the other inequality following from the definitions of these quantities for arbitrary maps.
\end{proof}

The results below show that the induced numerical diameter (whether restricted to self-adjoint arguments or not) tells us interesting things about linear maps.

There is a natural dichotomy between those maps with finite and infinite numerical diameter. This depends on what happens to the unit.
For the sake of brevity, we introduce the following terminology:
\begin{definition}
  We say that a map $\Phi: \mathcal S \to \mathcal S'$ between operator systems is \textbf{paraunital} if there is some $c \in \mathbb C$ for which $\Phi(1_{\mathcal S}) = c 1_{\mathcal S'}$.  
\end{definition}

\begin{theorem}\label{thm:paraunitalbounds}
  Let $\Phi: \mathcal S \to \mathcal S'$ be a bounded linear map between operator systems such that $\mathcal S\neq \mathbb C\cdot 1_{\mathcal S}$. The following are equivalent:
  \begin{enumerate}
      \item $\Phi$ is paraunital;
      \item $\lVert \Phi \rVert_{\mathrm{diam}} < \infty$;
      \item $\lVert \Phi \rVert_{\mathrm{sdiam}} < \infty$.
  \end{enumerate}
Furthermore, if $\Phi$ is paraunital, then $\|\Phi\|_{\diam} \leqslant 2\|\Phi\|$ and $\|\Phi\|_{\sdiam} \leqslant \|\Phi\|$.
\end{theorem}
\begin{proof}
  Suppose $\Phi$ is not paraunital, that is $\|\Phi(1_{\mathcal S})\|_{\diam} > 0$. Let $E \in \mathcal S_{sa}$ be independent of $1_{\mathcal S}$, for which $\lVert \Phi(E) \rVert_{\mathrm{diam}} > 0$. Note that if no such $E$ existed, there would exist $c,d\in \mathbb C$ such that
  \[
  \Phi(1_\mathcal S) = \Phi(E) + \Phi(1_\mathcal S - E) = c1_{\mathcal S'} + d1_{\mathcal S'},
  \]
  a contradiction.
  Let $\alpha = \| \Phi(E) \|_{\diam} \big/ \| \Phi(1_{\mathcal S}) \|_{\diam} > 0$.
  Then
  \begin{allowdisplaybreaks}%
  \begin{align*}
  \|\Phi\|_{\diam} \geqslant \|\Phi\|_{\sdiam} \ & \ \geqslant
     \lim_{n\rightarrow \infty} \frac{\|\Phi(1_{\mathcal S} + \frac{1}{n}E)\|_{\diam}}{\|1_\mathcal S + \frac{1}{n}E\|_{\diam}}
     \\[2ex] & \ = \lim_{n\rightarrow \infty} \frac{\|\Phi(1_{\mathcal S}) + \frac{1}{n}\Phi(E)\|_{\diam}}{\frac{1}{n}\|E\|_{\diam}}
     \\[1ex] & \ \geqslant \lim_{n\rightarrow \infty} \frac{\Big|\|\Phi(1_{\mathcal S})\|_{\diam} - \frac{1}{n}\|\Phi(E)\|_{\diam}\Big|}{\frac{1}{n}\|E\|_{\diam}}
     \\ & \ = \lim_{n\rightarrow \infty} \frac{\Bigl| (n - \alpha) \, \| \Phi(1_{\mathcal S})\|_{\diam} \Bigr|}{\|E\|_{\diam}}
     % \\[2ex] &
     \ = \lim_{n\rightarrow \infty} \, \lvert n - \alpha \rvert 
     \, \frac{\|\Phi(1_{\mathcal S})\|_{\diam}}{\|E\|_{\diam}} \ = \ \infty,
  \end{align*}%
  \end{allowdisplaybreaks}%
  where the the second inequality follows from the properties of the seminorm.

Conversely, 
  let $\Phi(1_{\mathcal S}) = {c 1_{\mathcal S'}}$ for $c\in\mathbb C$. Suppose $E\in \mathcal S$ with $\|E\|_{\diam} = 1$. 
  By Proposition \ref{prop:diameterlowerbounds} there exist $k_{\!\:\Re E},k_{\!\:\Im E}\in\mathbb R$ such that 
  \[
  \|E - (k_{\!\:\Re E} + i\, k_{\!\:\Im E})1_{\mathcal S}\| \leqslant \|E\|_{\diam} = 1\, .
  \]
  
  Hence, by letting $k = k_{\!\:\Re E} + i\, k_{\!\:\Im E}$ 
  \begin{align*}
      \frac{\|\Phi(E)\|_{\diam}}{\|E\|_{\diam}}
        \ & \ = \|\Phi(E)\|_{\diam}
      \\ \ & \ = \|\Phi(E) - ck1_{\mathcal S'}\|_{\diam}
      \\[1ex] & \ = \|\Phi(E - k1_{\mathcal S})\|_{\diam}
      \\[1ex] & \ \leqslant 2\|\Phi(E - k1_{\mathcal S})\|
      \\[1ex] & \ \leqslant 2\|\Phi\|\|E-k1_{\mathcal S}\| \leqslant 2\|\Phi\|.
  \end{align*}
 Therefore, if $\Phi$ is paraunital then $\|\Phi\|_{\diam} \leqslant 2\|\Phi\| < \infty$.
 
 Lastly, if $E$ is self-adjoint then Proposition \ref{prop:diameterlowerbounds} gives that there exists $k = k_E\in \mathbb R$ such that $\|E-k 1_{\mathcal S}\| = \frac{1}{2}\|E\|_{\diam} = \frac{1}{2}$. Inserting this into the above sequence of inequalities yields that if $\Phi$ is paraunital then $\|\Phi\|_{\sdiam} \leqslant \|\Phi\| < \infty$.
 \end{proof}

The following example shows that the numerical diameter can indeed be strictly bigger than the norm. 

\begin{example}\label{example:diambound}
Let $\Phi : M_2(\mathbb C) \rightarrow M_2(\mathbb C)$ be defined as
\[
\Phi\left(\left[\begin{matrix} a & b \\ c & d\end{matrix}\right]\right) \ = \ \left[\begin{matrix} \frac{1}{\sqrt 2}(a+b) & 0 \\ 0 & \frac{1}{\sqrt 2}(d-b)\end{matrix}\right]\, .
\]
\begin{itemize}
\item
If\, $\left\|\left[\begin{matrix} a & b \\ c & d\end{matrix}\right]\right\| = 1$, then 
\[
\left| \tfrac{1}{\sqrt 2}(a+b)\right|^2 \leqslant \frac{1}{2}\bigl(|a|^2 + |b|^2 + 2\bigl|\:\!\Re(\bar a b)\:\!\bigr|\bigr) \leqslant 1.
\]

\medskip\noindent
With a similar calculation, we may see that $\smash{\left\|\Phi\text{$\small\left(\left[\begin{matrix} a & b \\ c & d\end{matrix}\right]\right)$}\right\| \leqslant 1}$, with equality for the matrix \text{\footnotesize$\left[\begin{matrix} 1 & i \\ 0 & 0\end{matrix}\right]$}.

\item
    By contrast: by Lemma~\ref{lemma:2x2-matrix-unit-numerical-diam}, we have $ \left\|\left[\text{\small$\begin{matrix} 0 & 1 \\ 0 & 0\end{matrix}$}\right]\right\|_{\diam} =\; 1\,$, and
\[
    \left\|\Phi\left(\left[
        \begin{matrix} 0 & 1 \\ 0  & 0\end{matrix}
    \right]\right)\right\|_{\diam}
    =\;\; 
    \left\|\left[
        \begin{matrix} \frac{1}{\sqrt 2} & 0 \\ 0 & -\frac{1}{\sqrt 2}\end{matrix}
    \right]\right\|_{\diam} =\; \sqrt 2\,.
\]
\end{itemize}
Therefore, $\Phi$ is a paraunital linear map with $\|\Phi\| = 1 < \sqrt 2 \leqslant \|\Phi\|_{\diam}$.
\end{example}

\noindent
The theorem above singles out the paraunital maps as the ones to which the numerical diameter is a property of interest.
In particular,

\begin{proposition}\label{prop:seminorm}
  Both $\|\cdot\|_{\diam}$ and $\|\cdot\|_{\sdiam}$ are seminorms on paraunital maps between two fixed operator systems, but are not norms. 
  In particular: for $\Phi:\mathcal S \rightarrow \mathcal S'$ we have $\|\Phi\|_{\diam} = 0$ if and only if there exists a linear functional $\varphi:\mathcal S\rightarrow \mathbb C$ such that $\Phi = \varphi\cdot 1_{\mathcal S'}$.
  Furthermore, $\|\Phi\|_{\sdiam} = 0$ if and only if $\|\Phi\|_{\diam} = 0$.
\end{proposition}
\begin{proof}
  The numerical diameter is evidently real-valued, positive semidefinite, and homogeneous.
  It also satisfies the triangle inequality:
  \begin{align*}
  \begin{aligned}[b]
    \lVert \Phi + \Psi \rVert_{\mathrm{diam}}
    \;&=
    \sup_{E \in \mathcal S \setminus \mathbb C\cdot 1_{\mathcal S}} \frac{\lVert \Phi(E) + \Psi(E) \rVert_{\mathrm{diam}}}{\lVert E \rVert_{\mathrm{diam}}}
%  \\[1ex]&\leqslant
%    \sup_{\substack{E \in \mathrm{Herm}(\mathcal H) \\ \lVert E \rVert_{\mathrm{diam}} > 0}} \frac{\lVert \Phi(E) \rVert_{\mathrm{diam}} + \lVert \Psi(E) \rVert_{\mathrm{diam}}}{\lVert E \rVert_{\mathrm{diam}}}
  \\[1ex]&\leqslant
    \sup_{E \in \mathcal S \setminus \mathbb C\cdot 1_{\mathcal S}}
    \frac{\lVert \Phi(E) \rVert_{\mathrm{diam}}}{\lVert E \rVert_{\mathrm{diam}}}
    \;+
    \sup_{F \in \mathcal S \setminus \mathbb C\cdot 1_{\mathcal S}}
    \frac{\lVert \Psi(F) \rVert_{\mathrm{diam}}}{\lVert F \rVert_{\mathrm{diam}}}
  \\[1ex]&=\;
    \lVert \Phi \rVert_{\mathrm{diam}} + \lVert\Psi\rVert_{\mathrm{diam}}\,.
  \end{aligned}
  \end{align*}
The restriction to the self-adjoint operators is a seminorm for exactly the same reasons.
 
Now, suppose $\Phi = \varphi\cdot 1_{\mathcal S'}$ for some linear functional $\varphi$. This then implies that 
\[
\|\Phi(E)\|_{\diam} \,=\, \|\varphi(E) \cdot 1_{\mathcal S'}\|_{\diam} \,=\, |\varphi(E)| \cdot \|1_{\mathcal S'}\|_{\diam} \,=\, 0.
\]
That is $\|\Phi\|_{\diam} = 0$. Conversely, suppose $\|\Phi\|_{\diam}=0$. For every $E\in\mathcal S$ one has $\|\Phi(E)\|_{\diam}=0$ and so by Proposition \ref{prop:diamseminorm} there exists a scalar $c_E\in\mathbb C$ such that $\Phi(E) = c_E 1_{\mathcal S'}$. It is straightforward to conclude that $\varphi:\mathcal S\rightarrow \mathbb C$ defined by $\varphi(E) = c_E$ is the required linear functional.

Lastly, suppose $\mathcal S$ is an operator system.
Suppose $\|\Phi\|_{\sdiam} = 0$. Then for any $E\in\mathcal S$, we have
\begin{equation*}
    \bigl\lVert \Phi(E) \bigr\rVert_{\diam}
    \;\leqslant\;
    \bigl\lVert \Phi\bigl(\!\:\Re(E)\bigr)\bigr\rVert_{\diam} +\;
    \bigl\lVert \Phi\bigl(\!\:\Im(E)\bigr)\bigr\rVert_{\diam}
    \;=\; 0 
\end{equation*}
since $\Re(E), \Im(E) \in \mathcal S_{sa}$.
The converse follows from the definitions.
\end{proof}

%\begin{remark}
%One could extend these ideas to self-adjoint linear maps from $B(\mathcal H)$ to $B(\mathcal H')$. Because every element is the linear combination of two Hermitian elements the results would be similar but with a factor of 2.
%\end{remark}

\begin{lemma}\label{lem:submultiplicative}
The numerical diameter is submultiplicative with respect to composition of paraunital maps. In particular, if $\Phi:\mathcal S \rightarrow \mathcal S'$ and $\Psi:\mathcal S' \rightarrow \mathcal S''$ are linear paraunital maps between operator systems then
\[
\|\Psi\circ\Phi\|_{\diam} \leqslant \|\Psi\|_{\diam}\|\Phi\|_{\diam}.
\]
Moreover, The self-adjoint numerical diameter is submultiplicative with respect to composition of paraunital self-adjoint maps between operator systems.
\end{lemma}
\begin{proof}
If $\|\Psi\circ\Phi\|_{\diam} >0$ then there exists $E\in \mathcal S\setminus \mathbb C\cdot 1_{\mathcal S}$ such that $\|\Psi\circ\Phi(E)\|_{\diam} > 0$. If it was the case that $\|\Phi(E)\|_{\diam}=0$ then there exists $c\in\mathbb C$ such that $\Phi(E) = c1_{\mathcal S'}$, which gives that for some $d\in\mathbb C$
\[
\|\Psi\circ\Phi(E)\|_{\diam} = \|\Psi(c1_{\mathcal S'})\|_{\diam}
= \|cd1_{\mathcal S''}\|_{\diam} = 0,
\]
a contradiction. Thus, $\|\Phi(E)\|_{\diam}>0$ and 
\[
\frac{\|\Psi\circ\Phi(E)\|_{\diam}}{\|E\|_{\diam}} \  =\ \frac{\|\Psi\circ\Phi(E)\|_{\diam}}{\|\Phi(E)\|_{\diam}}\:\frac{\|\Phi(E)\|_{\diam}}{\|E\|_{\diam}}
\ \leqslant \ \|\Psi\|_{\diam}\|\Phi\|_{\diam}.
\]
The self-adjoint numerical diameter argument follows identically for self-adjoint maps.
\end{proof}

One can extend these results to some cases of non-paraunital maps but there will be problems when composing 0 and $\infty$ numerical diameter maps.

In the application of numerical diameters of linear maps, to the question of procedures to distinguish pairs of quantum states, we should consider the possibility that $\Phi$ may represent a section of the evolution on only one part, of a composite quantum system on which the observable $E$ is supported.
We might then ask whether the numerical diameter is  stable under tensor product with identity maps: that is, for unital $\Phi$, whether $\lVert \Phi \ox \id_{n} \lVert_{\diam} \,= \lVert \Phi \rVert_{\diam}$ for any finite-dimensional Hilbert space $\mathbb C^n$. This framework has been used to study super-operators and the Schatten p-norms \cite{Watrous-2005}.
In analogy to the completely bounded norm, we may define:
\begin{definition}
  \label{def:cbSpectralDiam}
  For a paraunital map $\Phi: \mathcal S \to \mathcal S'$ between operator systems, the \textbf{completely bounded numerical diameter} of $\Phi$ is given by 
  \[
  \lVert \Phi \rVert_{\cbdiam} = \sup_{n} \,\lVert \Phi \ox \id_{n} \rVert_{\diam}\,,
  \]
   taking the supremum over all finite-dimensional Hilbert spaces $\mathbb C^n$. Similarly, the \textbf{completely bounded self-adjoint numerical diameter} is defined to be  
   \[
   \|\Phi\|_{\cbsdiam} = \sup_{n} \,\lVert \Phi \ox \id_{n} \rVert_{\sdiam}.
   \]
\end{definition}

\begin{theorem}\label{thm:cbnorms}
    For paraunital linear maps between two fixed operator systems, $\lVert \cdot \rVert_{\cbsdiam}$ and $\lVert \cdot \rVert_{\cbdiam}$ are norms.
\end{theorem}
\begin{proof}
    That the completely bounded numerical diameter is homogeneous, positive semidefinite, and satisfies the triangle inequality all follow from Proposition \ref{prop:seminorm}.
    We show that the completely bounded numerical diameter is in fact definite.
    If $\Phi$ is non-zero, let $F$ be such that $\Phi(F) \ne 0$, and consider $G \neq c1_n$ for any $n > 1$.
    For $E = F \otimes G$, we then have
    \[
        (\Phi \otimes \id_n)(E) = \Phi(F) \otimes G
    \]
    which is not proportional to the identity, and in particular non-zero.
    It follows that ${\lVert \Phi(F) \otimes G \rVert_{\diam}} > 0$, so that $\lVert \Phi \rVert_{\cbdiam} > 0$.
    Similar remarks apply to the completely bounded self-adjoint numerical diameter by choosing $F$ and $G$ self-adjoint.    
\end{proof}
We can strengthen this observation, by describing how these norms relate to the completely bounded norm.

\begin{theorem}
  \label{thm:cbSpectralDiam}
  Let $\Phi: \mathcal S \to \mathcal S'$ be a paraunital map between operator systems. Then:
\smallskip
\begin{romanum}
\item
$ 
  \tfrac{1}{2}\lVert\Phi \rVert \leqslant \lVert \Phi\otimes \id_{4} \rVert_{\sdiam} \leqslant \lVert \Phi\otimes \id_{4} \rVert_{\diam}\,:
$
in particular,
\[
\tfrac{1}{2}\lVert \Phi \rVert_{\cb}  \leqslant \lVert \Phi \rVert_{\cbsdiam} \leqslant
\lVert \Phi \rVert_{\cb}
\quad \quad \textrm{and}  \quad  \quad
\tfrac{1}{2}\lVert \Phi \rVert_{\cb} \leqslant\lVert \Phi \rVert_{\cbdiam}\leqslant  2\lVert \Phi \rVert_{\cb} .
 \]
% \smallskip

\item
If $\Phi$ is self-adjoint then 
$
\lVert\Phi \rVert \leqslant \lVert \Phi\otimes \id_{4} \rVert_{\sdiam} = \lVert \Phi\otimes \id_{4} \rVert_{\diam} \,$:
in particular,
  \[
\lVert \Phi \rVert_{\cb} = \|\Phi\|_{\cbsdiam}  = \lVert \Phi \rVert_{\cbdiam} .
\]
\end{romanum}
\end{theorem}
\begin{proof}
Let $E\in \mathcal S$ be self-adjoint. Then $E\oplus (-E)\in M_2(\mathcal S)$ has symmetric spectrum around the origin. Hence,
\[
\|E\oplus (-E)\|_{\diam} = 2\|E\oplus (-E)\| = 2\|E\|.
\]
This implies that for a general $E\in \mathcal S$
\begin{align*}
2\|E\| \ & \leqslant \ 2\bigl\lVert \Re(E) \bigr\rVert + 2\bigl\lVert \Im(E) \bigr\rVert
\\[.5ex] & = \ \bigl\lVert\Re(E) \oplus \bigl(-\!\:\Re  (E)\bigr) \bigr\rVert_{\diam} +\; \bigl\lVert\Im(E) \oplus \bigl(-\!\:\Im(E)\bigr) \bigr\rVert_{\diam}
\\[.5ex] & = \ \bigl\lVert \Re(E \oplus (-E)) \bigr\rVert_{\diam} +\; \bigl\lVert \Im(E \oplus (-E))\bigr\rVert_{\diam}
\\[.5ex] & \leqslant \ 2\bigl\lVert E \oplus (-E) \bigr\rVert_{\diam}
\end{align*}
%First, suppose for an operator $A$ and a unitary $U$ we have $UAU^* = -A$.
%This implies that $-\!\:\Re   A = U(\!\:\Re   A) U^*$ and $-\!\:\Im   A = U(\!\:\Im   A)U^*$ and so 
%\[
%\|\!\:\Re   A\|_{\diam} = 2\|\!\:\Re   A\|\quad \textrm{and} \quad \|\!\:\Im   A\|_{\diam} = 2\|\!\:\Im   A\|.
%\]
%Hence,
%\begin{align*}
%%2\|A\|_{\diam} \ &  \geqslant\ \|\!\:\Re   A\|_{\diam} + \|\!\:\Im   A\|_{\diam} 
%%\\ & =\ 2\|\!\:\Re   A\| + 2\|\!\:\Im   A\|
%\\ & \geqslant \ 2\|A\|
%\end{align*}
Now to the main argument, first calculate
		\[
		  \begin{aligned}[b]
					\lVert \Phi\otimes \id_2 \rVert_{\sdiam}
				\;& \geqslant
					\sup_{E\in \mathcal S_{sa}, E\neq 0}
						\frac{\bigl\lVert (\Phi \ox \id_{2})\bigl(E \oplus (-E)\bigr) \bigr\rVert_{\diam}}{\bigl\lVert E \oplus (-E) \bigr\rVert_{\diam}}
				\\[.5ex]& =
					\sup_{E\in \mathcal S_{sa}, E\neq 0}
						\frac{\bigl\lVert \Phi(E) \oplus \bigl[-\Phi(E)\bigr] \bigr\rVert_{\diam}}{\bigl\lVert E \oplus (-E) \bigr\rVert_{\diam}}
				\\[.5ex]&\geqslant
					\sup_{E\in \mathcal S_{sa}, E\neq 0}
						\frac{\bigl\lVert \Phi(E) \bigr\rVert}{2\bigl\lVert E\oplus (-E) \bigr\rVert}
      % \\&
      \ \ \geqslant
					\sup_{E\in \mathcal S_{sa}, E\neq 0}
						\frac{\bigl\lVert \Phi(E) \bigr\rVert}{2\bigl\lVert E \bigr\rVert}\;,
			%	\;=\;
			%		 \lVert \Phi \rVert\,.
		  \end{aligned}
		\]
where the numerator inequality of the second last step is given by the first argument of this proof and the denominator inequality is given by Proposition \ref{prop:diamcontinuous}.
Next, we have
\begin{align*}
\sup_{E\in M_2(\mathcal S)_{sa}, E\neq 0} \frac{\|(\Phi\otimes \id_2)(E)\|}{\|E\|} \ &  \geqslant \ \sup_{A\in \mathcal S, A\neq 0} \frac{\left\| (\Phi\otimes \id_2)\left(\left[\begin{matrix}0 & A \\ A^* & 0\end{matrix}\right] \right)\right\|}{\left\| \left[\begin{matrix}0 & A \\ A^* & 0\end{matrix}\right]\right\|}
\\[2ex] & = \ \sup_{A\in \mathcal S, A\neq 0} \frac{\max\{\Phi(A), \Phi(A^*)\}}{\|A\|}
% \\ &
\ =\ \|\Phi\|.
\end{align*}
Combining these two step yields $\frac{1}{2}\lVert\Phi \rVert \leqslant \lVert \Phi\otimes \id_{4} \rVert_{\sdiam}$ by replacing $\Phi$ with $\Phi\otimes\id_2$ in the first set of calculations.
Combining this with Theorem \ref{thm:paraunitalbounds} gives for any $n\in\mathbb N$
\begin{align*}
\\[-3ex]
 &\tfrac{1}{2}\|\Phi\otimes \id_n\| \leqslant \lVert \Phi \ox \id_{4n} \rVert_{\diam} \leqslant 2\lVert \Phi\otimes \id_{4n} \rVert
\quad  \textrm{and}  \\[.5ex]
 &\tfrac{1}{2}\|\Phi\otimes \id_n\| \leqslant\lVert \Phi \ox \id_{4n} \rVert_{\sdiam} \leqslant \lVert \Phi\otimes \id_{4n} \rVert\, .
\\[-3ex]
\end{align*}
This proves part \parit{i}.
For \parit{ii}, we add the further assumption that $\Phi$ is self-adjoint, and obtain
\[
		  \begin{aligned}[b]
					\lVert \Phi\otimes \id_2 \rVert_{\sdiam}
				\;& \geqslant
					\sup_{E\in \mathcal S_{sa}, E\neq 0}
						\frac{\bigl\lVert (\Phi \ox \id_{2})\bigl(E \oplus (-E)\bigr) \bigr\rVert_{\diam}}{\bigl\lVert E \oplus (-E) \bigr\rVert_{\diam}}
				\\[2ex]& =
					\sup_{E\in \mathcal S_{sa}, E\neq 0}
						\frac{\bigl\lVert \Phi(E) \oplus \bigl[-\Phi(E)\bigr] \bigr\rVert_{\diam}}{\bigl\lVert E \oplus (-E) \bigr\rVert_{\diam}}
				% \\&
    \;\;\geqslant
					\sup_{E\in \mathcal S_{sa}, E\neq 0}
						\frac{2\bigl\lVert \Phi(E) \bigr\rVert}{2\bigl\lVert E \bigr\rVert}.
			%	\;=\;
			%		 \lVert \Phi \rVert\,.
		  \end{aligned}
		\]
		since $ \Phi(E) \oplus \bigl[-\Phi(E)\bigr]$ is self-adjoint, which allows us to use the very first equality of this proof. Thus, $\lVert\Phi \rVert \leqslant \lVert \Phi\otimes \id_{4} \rVert_{\sdiam}$, which by Proposition~\ref{prop:induced-sa-diam-on-sa-maps} is also equal to $\lVert \Phi\otimes \id_{4} \rVert_{\diam}$.
Therefore, using the upper bounds of the general version yields the last required result.
\end{proof}

% In the previous proof it was established that for $\Phi$ self-adjoint,
% \[
% \|\Phi\|_{\operatorname{cb}} = \sup_{n} \left\|(\Phi\otimes \id_n)|_{M_n(\mathcal S)_{sa}}\right\|_{\sdiam}\, .
% \]

Note that the (completely bounded) self-adjoint numerical diameter can in fact be strictly smaller than the (completely bounded) norm, for non-self-adjoint maps. This is seen in the following example.

\begin{example}
Consider the map $\Phi(\mathbb C) : M_2\rightarrow M_2(\mathbb C)$ defined by
\[
\Phi\left(\left[\begin{matrix} a & b \\ c & d\end{matrix}\right]\right) \ = \ \left[\begin{matrix} 0 & b \\ 0 & 0\end{matrix}\right].
\]
Consider the analogous map $\Phi_n : M_2(M_n(\mathbb C)) \rightarrow M_2(M_n(\mathbb C))$ defined by
\[
\Phi_n\left( \left[\begin{matrix} A & B \\ C &D\end{matrix}\right]\right) = \left[\begin{matrix} 0 & B \\ 0 &0\end{matrix}\right]:
\]
observe that there exists a permutation $P_n \in M_{2n}(\mathbb C)$, called the canonical shuffle, which defines an isometry $U(E) = P_n^*EP_n$ such that \[
U \circ (\Phi\otimes \id_n) = \Phi_n \,.
\]
It follows that $\|\Phi\| = \|\Phi\|_{\cb} = 1$.
%It should be relatively straightforward to prove that $\|\Phi_n\|_{\sdiam} = 1/2$ for all $n\geq 1$ and so $\|\Phi\|_{\cbsdiam} = 1/2$.

To contrast, we may show that $\lVert \Phi \rVert_{\cbsdiam} < 1$.
Consider a self-adjoint matrix
\[
    E = \left[\begin{matrix} A & B \\ B^* & D\end{matrix}\right] \in M_2(M_n(\mathbb C)).
\]
Without affecting $\lVert E \rVert_{\diam}$\,, we may restrict our interest to the case $\lambda_{min}(E) = 0$ (so that $E \geqslant 0$), as we may otherwise consider the operator $E' = E \,-\, \lambda_{\min}(E)\cdot 1_{2n}$.
From $E\geqslant 0$, it follows that for $x,y\in \mathbb C^n$, we have
\[
0 \,\leqslant\, \begin{bmatrix} x & 0 \\ 0 & y \end{bmatrix}^*\! E\;\! \begin{bmatrix} x & 0 \\ 0 & y \end{bmatrix} \,=\, \left[\begin{matrix} x^*Ax & x^*By \\ y^*B^*x & y^*Dy\end{matrix}\right] \in M_2(\mathbb C).
\]
The non-negativity of the determinant gives that $|x^*By|^2 \leqslant (x^*Ax)(y^*Dy)$, which in turn implies that $|x^*By| \leqslant x^*Ax + y^*Dy$.
If we consider unit vectors $x,y \in \C^n$ for which $x^*By = \|B\|$, we obtain 
\[
\left\langle E\left[\begin{matrix} x \\ y\end{matrix}\right], \left[\begin{matrix} x \\ y\end{matrix}\right]\right\rangle \,=\, x^*Ax + x^*By + y^*B^*x + y^*Dy \,\geqslant\, 4\|B\|.
\]
As the norm of the block vector above is $\sqrt 2$, it follows that $2\|B\| \in W(E)$. We also assumed that $0\in W(E)$ and so $\|E\|_{\diam} \geqslant 2\|B\|$.
On the other hand, extending the proof of Lemma~\ref{lemma:2x2-matrix-unit-numerical-diam} to block matrices gives us
\[
W(\Phi_n(E)) \,=\, W\left(\left[\begin{matrix} 0 & B \\ 0 & 0 \end{matrix}\right]\right) \,=\, \left\{ z\in \mathbb C : |z| \leqslant \tfrac{1}{2}\|B\|\right\}
\]
This gives that $\|\Phi_n(E)\|_{\diam} = \|B\|$ and so $\|\Phi_n\|_{\sdiam} \leqslant 1/2$.
Considering $E  = \left[\begin{matrix} 0 & I \\ I & 0\end{matrix}\right]$ as an argument to $\Phi$, we may see that this bound holds with equality.
As unitary similarity doesn't change numerical diameters,
\[
\|\Phi\|_{\sdiam} = \|\Phi\|_{\cbsdiam} = 1/2 < 1 = \|\Phi\| = \|\Phi\|_{\cb}.
\]
\end{example}

This gives us the following conclusion:

\begin{theorem}\label{thm:cbdiamnorm}
For paraunital linear maps between fixed operator systems, $\|\cdot\|_{\cbdiam}$, $\| \cdot \|_{\cbsdiam}$, and $\|\cdot\|_{\cb}$ are all comparable yet distinct norms, but which coincide on self-adjoint maps.
\end{theorem}
\begin{proof}
The previous theorem gives the comparability of the norms and the previous example demonstrated that $\|\Phi\|_{\cbsdiam}$ can be strictly smaller than $\|\Phi\|_{\cb}$. 

Now consider the map $\Phi$ of Example~\ref{example:diambound}: as it maps into a commutative C$^*$-algebra,
we can apply \cite[Theorem 3.9]{Paulsenbook}, that the cb-norm of $\Phi$ is not larger than its operator norm.
Then 
\[
\|\Phi\|_{\cbsdiam} \leqslant \|\Phi\|_{\cb} = \|\Phi\| = 1 < \sqrt 2 \leqslant \|\Phi\|_{\diam} \leqslant \|\Phi\|_{\cbdiam}\, .
\]
Therefore, all three norms are distinct.
The statement regarding the application of these norms to self-adjoint paraunital maps immediately follows from Theorem~\ref{thm:cbSpectralDiam} and from the definition of the self-adjoint numerical diameter.
\end{proof}

An important consequence of Theorem \ref{thm:cbSpectralDiam} is that we now have examples of linear maps whose completely bounded numerical diameter is strictly greater than its numerical diameter and this can be achieved without having to actually calculate any specific diameters.

\begin{example}\label{ex:growingcbdiam}
Consider the transpose map $\Phi : M_n(\mathbb C) \rightarrow M_n(\mathbb C)$, $\Phi(A) = A^T$. Since $\|A^T\|_{\diam} = \|A\|$, because $W(A^T) = W(A)$, then $\|\Phi\|_{\diam} = 1 = \|\Phi\|$. It is a well-known result of Tomiyama \cite{Tomiyama} that $\|\Phi\|_{\cb} = n$. Therefore, the previous theorem gives us that $\|\Phi\|_{\cbdiam} \geqslant \|\Phi\|_{\cb} = n > 1 = \|\Phi\|_{\diam}$.
\end{example}

%\vskip 18 pt
%\stepcounter{section}
%\noindent {\bf\thesection.~Reduction to unital positive maps and their sections.}
\section{Unital (completely) positive maps and their sections}

Among the paraunital maps, we are most interested in maps which are actually unital, and specifically with those which are either completely positive (`ucp' maps), or which are \emph{(bounded) sections} (\emph{i.e.},~right-inverses) of ucp maps.
It is clear that results on unital maps will often yield related results for paraunital maps; below we describe how, for injective paraunital maps $\Phi$, and for maps $\Phi$ whose null-space is spanned by $1$, we may to an extent reduce properties of the numerical diameter to the case of maps $\tilde \Phi$ which are unital completely positive maps, or sections of such maps.

\subsection{Preliminaries on (completely bounded) expansive maps}

It is good to first look at what is and is not possible for unital positive maps.

\begin{proposition}[Corollary 2.8, Proposition 2.11 \cite{Paulsenbook}]\label{prop:poscontractive}
  Let $\Phi: \mathcal S \to \mathcal S'$ be a unital linear map between operator systems.
  If $\Phi$ is contractive then it is positive. Moreover, if $\mathcal S$ is a C$^*$-algebra and $\Phi$ is positive then it is contractive.
\end{proposition}

\begin{definition}
Suppose $\Phi : \mathcal S \rightarrow \mathcal S'$ is a linear map between operator systems. We say $\Phi$ is \textbf{expansive} if $\|\Phi(A)\| \geqslant \|A\|$ for every $A\in \mathcal S$. Moreover, we say $\Phi$ is \textbf{completely expansive} if $\Phi\otimes \id_n$ is expansive on $\mathcal S\otimes M_n(\mathbb C)$ for every $n\in\mathbb N$.
\end{definition}

\begin{proposition}\label{prop:expansive}
  Let $\Phi: \mathcal S \to \mathfrak B$ be  a unital bounded map from an operator system to a C$^*$-algebra.
  If $\Phi$ is the section of a positive map then it is expansive.
\end{proposition}
\begin{proof}
  If $\Phi$ is the section of a positive map $\Psi: \mathfrak B \to \mathcal S$, then $\Psi$ is itself a unital map, and so is contractive by the Proposition \ref{prop:poscontractive}.
  For any $E \in \mathcal S$ with $\lVert E \rVert = 1$, let $\tilde E = E / \lVert \Phi(E) \rVert$, and $F = \Phi(\tilde E)$.
  By construction, we have $\lVert F \rVert = 1$, so that
\[
  \begin{aligned}
    \frac{1}{\lVert \Phi(E) \rVert} 
  =
    \lVert \tilde E \rVert
  =
    \lVert \Psi(F) \rVert
  \leqslant
    1\,,
  \end{aligned}
\]
  so that $\lVert \Phi(E) \rVert \geqslant 1$.
 \end{proof}

There is no converse to the previous proposition in general.
To see this, suppose one has a unital bounded expansive map.
Then, it is injective and its inverse on the range operator system will be unital and contractive, and therefore positive by Proposition \ref{prop:poscontractive}. However, unital positive maps on operator systems need not extend to positive maps on the whole algebra. There are a number of examples of this failure of extension. Below we see an example of this in matrices.

\begin{example}[Paulsen \cite{Paulsen}]
Consider the operator system 
\[
\mathcal S \ = \ \left\{ \left[\begin{matrix} aI_2 & B \\ C & dI_2 \end{matrix}\right] : a,d\in \mathbb C, B,C\in M_2(\mathbb C)\right\} \subseteq M_4(\mathbb C)
\]
and consider the map $\Phi : \mathcal S \rightarrow M_4(\mathbb C)$ defined by 
\[
\Phi\left(\left[\begin{matrix} aI_2 & B \\ C & dI_2 \end{matrix}\right]\right) \ = \ \left[\begin{matrix} aI_2 & B^T \\ C^T & dI_2 \end{matrix}\right]. 
\]
Any nondiagonal positive element of $\mathcal S$ must have $a,d>0$. Hence, 
\bgroup
\allowdisplaybreaks
\begin{align*}
    \left[\begin{matrix} 
        aI_2 & B \\[1ex] B^* & dI_2 
    \end{matrix}\right] \geqslant 0 &
    \ \iff \ 
    \left[\begin{matrix}
        I_2 & \frac{1}{a}B \\[1ex] \frac{1}{a}B^* & \frac{d}{a}I_2
    \end{matrix}\right] \geqslant 0
\\[1ex] & \ \iff \ 
    \frac{1}{a^2}B^*B \leqslant \frac{d}{a}I_2  \;,
\qquad \textrm{by \cite[Lemma 3.1]{Paulsenbook}}
\\[1ex] & \ \iff \ 
    \|B^T\| = \|B\| \leqslant \sqrt{ad}
\\[.5ex] & \ \iff \ 
    \frac{1}{a^2}(B^T)^*B^T \leqslant \frac{d}{a}I_2
\\[1ex] & \ \iff \ 
    \Phi\left(\left[\begin{matrix} 
        aI_2 & B \\ B^* & dI_2 
    \end{matrix}\right]\right)
    \ = \ 
    \left[\begin{matrix} 
        aI_2 & B^T \\ (B^*)^T & dI_2 
    \end{matrix}\right]\geqslant 0\;.
\end{align*}%
\egroup%
Thus, $\Phi$ is unital and positive. 
Suppose that $\Phi$ had a positive extension to all of $M_4(\mathbb C)$. We may calculate:
\begin{align*}
\mathcal S 
\,\ni\, 
\left[\begin{matrix} I_2 & 0 \\ 0 & 0 \end{matrix}\right]
\  = \ \Phi\left(\left[\begin{matrix} I_2 & 0 \\ 0 & 0 \end{matrix}\right]\right)
 & = \ \Phi\left(\left[\begin{matrix} E_{1,1} & 0 \\ 0 & 0 \end{matrix}\right]\right) + \Phi\left(\left[\begin{matrix} E_{2,2} & 0 \\ 0 & 0 \end{matrix}\right]\right)
\\[.5ex] & = \ \left[\begin{matrix} X_1 & Y_1 \\ Y_1^* & Z_1 \end{matrix}\right] + \left[\begin{matrix} X_2 & Y_2 \\ Y_2^* & Z_2 \end{matrix}\right]
\end{align*}
with both of these latter matrices positive. This means that $Z_1 = -Z_2$ and $Z_1,Z_2\geqslant 0$ which implies that $Z_1=Z_2=0$ and $Y_1 = Y_2 = 0$.
Hence
\[
\Phi\left(\left[\begin{matrix} E_{1,1} & E_{1,1} \\ E_{1,1} & I_2 \end{matrix}\right]\right) \ = \ \left[\begin{matrix} X_1 & E_{1,1} \\ E_{1,1} & I_2 \end{matrix}\right]\geqslant 0
\ \ \textrm{and} \ \
\Phi\left(\left[\begin{matrix} E_{2,2} & E_{2,1} \\ E_{1,2} & I_2 \end{matrix}\right]\right) \ = \ \left[\begin{matrix} X_2 & E_{1,2} \\ E_{2,1} & I_2 \end{matrix}\right]\geqslant 0
\]
since the matrix outside of the upper left-hand corner is in $\mathcal S$. This implies that \[E_{1,1}X_1E_{1,1}, E_{1,1}X_2E_{1,1} \geqslant 1\]
and so
\[
1 = E_{1,1}I_2E_{1,1} = E_{1,1}X_1E_{1,1} + E_{1,1}X_2E_{1,1} \geqslant 2
\]
a contradiction. Therefore, $\Phi$ is a unital positive linear map of an operator system that does not extend to a positive map on a C$^*$-algebra.
\end{example}

%\begin{example}
%Consider the matrix $A = \left[\begin{matrix}& \left[\begin{matrix}1& %1\\ 0&1\end{matrix}\right]
%\\ \left[\begin{matrix}0& 1\\ 0&0\end{matrix}\right]\end{matrix}\right]$
%then it is claimed that \\ $C^*(I_4, A) = M_4$.
%\end{example}

We now repeat the previous discussion, but in the completely positive context, since the complete structure is very powerful. In particular, the assumption about mapping into a C$^*$-algebra can be weakened.

\begin{proposition}\label{prop:cccpcexpansive}
  Let $\Phi : \mathcal S \rightarrow \mathcal S'$ be a unital linear map between operator systems. Then $\Phi$ is completely positive if and only if it is completely contractive ($\|\Phi\|_{\textrm{cb}} = 1$). Moreover, if it is the completely bounded section of a completely positive map then it is completely expansive.
\end{proposition}
\begin{proof}
The first statement is \cite[Proposition 3.5-3.6]{Paulsenbook}. The second statement follows from adjusting the proof of Proposition \ref{prop:expansive} to this complete context. 
\end{proof}

\begin{proposition}
Let $\Phi: B(\mathcal H) \to \mathcal S'$ be a unital completely bounded linear map into an operator system.
If $\Phi$ is completely expansive then it is the section of a completely positive map.
\end{proposition}
\begin{proof}
First, concretely embed $\mathcal S' \subseteq B(\mathcal H')$.
Suppose $\Phi$ is completely expansive which gives that it is injective. Let $\mathcal T' = \Phi(B(\mathcal H)) \subseteq \mathcal S'$, an operator system. Define $\psi : \mathcal T' \rightarrow B(\mathcal H)$ to be the inverse of $\Phi$. It is immediate that $\psi$ is unital and completely contractive which means that it extends to a unital completely positive map $\Psi$ on all of $B(\mathcal H')$ by Arveson's Extension Theorem. Therefore, $\Phi$ is a section of this $\Psi|_{\mathcal S'}$.
\end{proof}

We are really only interested in sections of unital positive maps, that are themselves unital, and self-adjoint.
These properties are not automatic, as for instance ${\Phi: \mathbb C\rightarrow M_2(\mathbb C)}$ defined by
\[
\Phi(x) = \left[\begin{matrix} 2x & x \\ 0 & 0\end{matrix}\right]
\]
is neither unital or self-adjoint but is a section of the normalized trace which is unital and positive.

\subsection{Diameters of bounded sections of unital positive maps}

We turn now to a discussion of the numerical diameter under unital positive maps and their bounded sections, with an aim to characterise isometries.

As positive maps $\Psi$ are self-adjoint, we have $\lVert \Psi \rVert_{\sdiam} = \lVert \Psi \rVert_{\diam}$ by Proposition~\ref{prop:induced-sa-diam-on-sa-maps} in this case; and similarly for any self-adjoint section that $\Psi$ admits.
We make use of this relationship without further comment for the remainder of this Section.

\begin{proposition}\label{prop:contractivediameter}
Let $\Phi : \mathcal S \rightarrow \mathcal S'$ be a unital (completely) bounded map between operator systems such that $\dim \mathcal S > 1$.
If $\Phi$ is positive, then $\|\Phi\|_{\diam} \leqslant 1$. Moreover, if $\Phi$ is completely positive then $\|\Phi\|_{\cbdiam} \leqslant 1$.
\end{proposition}
\begin{proof}
First assume $\Phi$ is positive. Concretely, and non-degenerately, embed $\mathcal S \subseteq B(\mathcal H)$ and $\mathcal S' \subseteq B(\mathcal H')$. 

Let $E\in \mathcal S$ such that it is not a scalar multiple of the identity. By definition, for each $z\in W(\Phi(E))$ there is a state $f$ of $C^*(1_\mathcal H,\Phi(E))$ such that $f(\Phi(E)) = z$. This implies that $f$ is unital and positive on
$\,\Span\,\{1_{\mathcal H'}, \Phi(E), \Phi(E)^*\}\,$. Hence, $f\circ\Phi$ is unital and positive on $\{1_{\mathcal H}, E, E^*\}$ and so completely positive \cite[Proposition 3.8]{Paulsenbook}. This extends to a ucp map $g$ on all of $C^*(1_\mathcal H, E)$ by Arveson's Extension Theorem. Thus, $g$ is a state and $z = f\circ \Phi(E) = g(E) \in \overline{W(E)}$. 

We have therefore proven $W(\Phi(E)) \subseteq \overline{W(E)}$.
The first result follows from this, and the second result follows by considering $\Phi \otimes \id_n$ for arbitrary $n > 1$.
\end{proof}

To illustrate the differences of Proposition \ref{prop:poscontractive} and the previous proposition consider the following example.

\begin{example}
Let $\mathcal S = \Span\{1, z, \bar z\} \subset C(\mathbb T)$ and define $\Phi : \mathcal S \rightarrow M_2(\mathbb C)$ by 
\[
\Phi(a + bz + c\bar z) \ = \ \left[\begin{matrix} a & 2b \\ 2c & a\end{matrix}\right].
\]
It can be shown that $\Phi$ is unital and positive but not contractive \cite[Example 2.2]{Paulsenbook}, in particular $\|\Phi\| = 2$. However, by the last proposition, $\Phi$ is still contractive with respect to the numerical diameter. In particular, $W(z)$ is the convex hull of $\sigma(z) = \mathbb T$ and so
\[
W\left(\left[\begin{matrix} 0 & 2 \\ 0 & 0\end{matrix}\right]\right)
= \{ 2y\bar x : |x|^2 + |y|^2 = 1\} = \overline{\mathbb D} = W(z).
\]
\end{example}

\begin{corollary}\label{cor:expansive}
If $\Phi : \mathcal S \rightarrow \mathcal S'$, $\dim \mathcal S > 1$, is a bounded section of a unital (completely) positive linear map, then $\Phi$ is (completely) expansive with respect to $\|\cdot\|_{\diam}$.
\end{corollary}

We can get a little finer detail by adding the assumption that the map is self-adjoint.

\begin{proposition}\label{prop:sectionisexpansive}
Let $\Phi: \mathcal S \to \mathcal S', \dim \mathcal S > 1$, be a unital self-adjoint bounded section of a unital positive map between operator systems, concretely embedded in $B(\mathcal H)$ and $B(\mathcal H')$ respectively.
 Then for all self-adjoint $A\in\mathcal S$,
  \[
 \lambda_{\min}\bigl(\Phi(A)\bigr) \leqslant \lambda_{\min}(A) \leqslant \lambda_{\max}(A) \leqslant \lambda_{\max}(\Phi(A)).
  \] 
\end{proposition}
\begin{proof}
Let $\Psi: \mathcal S' \to \mathcal S$ be a unital positive map for which $\Phi$ is a section. Let $A\in \mathcal S$ be self-adjoint and $E = {A - \lambda_{\min}(A)1_{\mathcal S}}$. Then $\lambda_{\min}(E) = 0$. If $\lambda_{\min}(\Phi(E)) > 0$ then
$\Phi(E) - \lambda_{\min}(\Phi(E))1_{\mathfrak B}$ is a positive operator
but
\[
\lambda_{\min}\left(\Psi\Big(\Phi(E) - \lambda_{\min}(\Phi(E))1_{\mathfrak B}\Big)\right) = \lambda_{\min}(E - \lambda_{\min}(\Phi(E))1_{\mathcal S}) < 0
\]
which contradicts the positivity of $\Psi$. Hence, $\lambda_{\min}(\Phi(E)) < 0 = \lambda_{\min}(E)$ and so 
\begin{align*}
\lambda_{\min}(\Phi(A)) \ & = \ \lambda_{\min}(\Phi(E + \lambda_{\min}(A)1_{\mathcal S})) 
\\ & = \ \lambda_{\min}(\Phi(E)) + \lambda_{\min}(A)
\\ & \leqslant \ \lambda_{\min}(A)\, .
\end{align*}
Lastly,
\[
\lambda_{\max}(A) = -\lambda_{\min}(-A) \leqslant -\lambda_{\min}(-\Phi(A)) = \lambda_{\max}(\Phi(A))\, .
\qedhere
\]
\end{proof}
\smallskip

\begin{proposition}%
    \label{prop:optorNorm-sdiam-ineqs}%
  Let $\Phi: \mathcal S \to \mathcal S'$ be a unital self-adjoint section of a unital positive map between operator systems.
  Then 
  \[
  \|\Phi\|_{\sdiam}
  \;\le\;
  \bigl\lVert\;\!\Phi\big\vert_{\mathcal S_{sa}}\;\!\bigr\rVert  \ \leqslant \ 2\|\Phi\|_{\sdiam} - 1
  \]
  and
 \[
  1 \;\leqslant\; \|\Phi\|_{\sdiam} \; \leqslant \; \lVert \Phi \rVert \; \leqslant \; 4\lVert \Phi \rVert_{\mathrm{sdiam}} - 2\,.
 \]
\end{proposition}
\begin{proof}
Let $\|\Phi\|_{\sdiam} = 1+\epsilon$, for $\epsilon\geqslant 0$, since we know $\Phi$ is expansive.
For $A \in \mathcal S$ self-adjoint with $\|A\|_{\diam} = 1$ we know by the previous proposition that 
\[
 \lambda_{\min}\bigl(\Phi(A)\bigr) \leqslant \lambda_{\min}(A) \leqslant \lambda_{\max}(A) \leqslant \lambda_{\max}(\Phi(A))\,.
  \]
Since $\lambda_{\max}(A)-\lambda_{\min}(A) = 1$ and $\lambda_{\max}(\Phi(A)) - \lambda_{\min}(\Phi(A)) \leqslant 1+\epsilon$ then
\[ \lambda_{\min}(A)-\epsilon \leqslant \lambda_{\min}(\Phi(A)) \quad \textrm{and} 
\quad \lambda_{\max}(\Phi(A)) \leqslant \lambda_{\max}(A) + \epsilon\,.
\]
Thus,
\begin{align*}
\|\Phi(A)\| \ & = \ \max\{\lambda_{\max}(\Phi(A)), -\lambda_{\min}(\Phi(A))\}
\\ & \leqslant \ \max\{\lambda_{\max}(A) + \epsilon, -\lambda_{\min}(A) + \epsilon\}
% \\ & 
\,=\, \|A\| + \epsilon\,
\end{align*}
which implies
\[
\frac{\|\Phi(A)\|}{\|A\|} \leqslant 1 + \frac{\epsilon}{\|A\|} \leqslant 1 + 2\epsilon
\]
since $\frac{1}{2} \leqslant \|A\| \leqslant 1$. Therefore, 
\[
\bigl\lVert \Phi |_{\mathcal S_{sa}}\bigr\rVert \,\leqslant\, 1 + 2\epsilon \,=\, 2(1+\epsilon) - 1 \,=\, 2\|\Phi\|_{\sdiam} - 1\, .
\]

Lastly, for any $A\in\mathcal S$
\begin{align*}
\|\Phi(A)\| \ & \leqslant \ \|\Phi(\!\:\Re   A)\| + \|\Phi(\!\:\Im   A)\|
\\[.5ex] & \leqslant \, \ \|\Phi|_{\mathcal S_{sa}}\|\cdot \|\Re   A\| \,+\, \|\Phi|_{\mathcal S_{sa}}\|\cdot \|\Im   A\|
\\[.5ex] & \leqslant\, \ 2\bigl\|\Phi|_{\mathcal S_{sa}}\bigr\|\cdot \|A\|
\\[.5ex] & \leqslant\, \ 2(2\|\Phi\|_{\sdiam}- 1)\|A\| 
\;\leqslant\, \ (4\|\Phi\|_{\sdiam} - 2)\|A\|\, .
\end{align*}
This along with Corollary~\ref{cor:expansive} and Theorem~\ref{thm:paraunitalbounds} establish the last set of inequalities.
\end{proof}

We can now reach our main conclusions of this section.

\begin{theorem}
Let $\Phi: \mathfrak A \to \mathfrak B$ be a unital self-adjoint bounded section of a unital positive map between unital C$^*$-algebras.
Then $\|\Phi\|_{\sdiam} = 1$ if and only if $\Phi$ is an isometric map.
\end{theorem}
\begin{proof}
Suppose $\|\Phi\|_{\sdiam} = 1$ which means that $\|\Phi(A)\|_{\diam} = \|A\|_{\diam}$ for every $A\in\mathfrak A$ self-adjoint by Corollary \ref{cor:expansive}.
More specifically, Proposition \ref{prop:sectionisexpansive} proves that $\lambda_{\max}(\Phi(A)) = \lambda_{\max}(A)$ and $\lambda_{\min}(\Phi(A))=\lambda_{\min}(A)$. Hence, $\Phi$ is a positive map and since it is also unital then $\Phi$ is contractive by Proposition \ref{prop:poscontractive}. Therefore, $\Phi$ is expansive and contractive, that is isometric.

On the other hand, suppose $\Phi$ is isometric. $\Phi$ is unital contractive and so it is positive. By Proposition \ref{prop:contractivediameter} $\|\Phi\|_{\diam} \leqslant 1$.
Combining this with the expansiveness of the numerical diameter we get $\|\Phi\|_{\diam} = \|\Phi\|_{\sdiam} = 1$.
\end{proof}

In the complete case we can loosen our context to operator systems.

\begin{theorem}\label{thm:cbdiamisometry}
Let $\Phi: \mathcal S \to \mathcal S'$ be a unital self-adjoint bounded section of a unital completely positive map between operator systems. Then $\|\Phi\|_{\cbsdiam} = 1$ if and only if $\Phi$ is a completely isometric map.
\end{theorem}
\begin{proof}
Suppose $\|\Phi\|_{\cbsdiam} = 1$. By Theorem \ref{thm:cbSpectralDiam} this implies that $\|\Phi\|_{\cb} \leqslant 1$, that is completely contractive. By Proposition \ref{prop:cccpcexpansive} $\Phi$ is also completely expansive and thus $\Phi$ is completely isometric.

On the other hand, suppose $\Phi$ is completely isometric. By Theorem \ref{thm:cbSpectralDiam} the numerical diameter of $\Phi$ is completely bounded and so by Proposition \ref{prop:contractivediameter} $\|\Phi\|_{\cbdiam} \leqslant 1$. Lastly, Corollary \ref{cor:expansive} gives that $\|\Phi\|_{\cbsdiam}\geqslant 1$, which finishes off the proof.
\end{proof}

\subsection{Approximately isometric sections}

It is desirable to have approximate versions of the above results, that having a (completely bounded) numerical diameter close to 1 implies that the map is close to an (completely) isometric map. This seemingly has not been studied except for the bijective case, which could indicate that it is difficult, which is hardly surprising since whole books have been written about isometric maps~(cf.~\cite{FlemingJamison}). As such we summarize what is known and give some partial results in our context.

A Jordan homomorphism $\psi : \mathcal A \rightarrow \mathcal B$ between C$^*$-algebras is a linear map such that
\[
\psi(ab + ba) = \psi(a)\psi(b) + \psi(b)\psi(a), \quad \forall a,b\in \mathcal A.
\]
In particular, homomorphisms and anti-homomorphisms are Jordan homomorphisms. A Jordan $*$-homomorphism is additionally a self-adjoint map and a Jordan isomorphism is additionally bijective.

\begin{theorem}[Theorem 7, Kadison \cite{Kadison}]
Let $\psi : \mathcal A_1\rightarrow \mathcal A_2$ be an isometric bijective linear map with $\mathcal A_2\subseteq B(\mathcal H_2)$. Then there exists a Jordan isomorphism $\varphi:\mathcal A_1 \rightarrow \mathcal A_2$ and a unitary $u\in B(\mathcal H_2)$ such that
$\psi = u\phi$.
\end{theorem}

\begin{theorem}[Theorem 10 and Corollary 11, Kadison \cite{Kadison}]
An isometric self-adjoint bijective linear map $\psi : \mathcal A_1\rightarrow \mathcal A_2$ between C$^*$-algebras is the sum of a $*$-isomorphism and a $*$-anti-isomorphism. 

Moreover, if $\mathcal A_1 = B(\mathcal H_1)$ and $\mathcal A_2 = B(\mathcal H_2)$ then there exists a unitary $u\in B(H_1,H_2)$ such that
\[
\psi(a) = uau^*, \forall a\in B(\mathcal H_1) \quad \textrm{or} \quad \psi(a) = ua^tu^*, \forall a\in B(\mathcal H_1)
\]
\end{theorem}

Note that the second conclusion is true for the broader class of factors, but we use the version above for its more relevant context.

\begin{corollary}
A completely isometric self-adjoint bijective linear map $\psi : \mathcal A_1\rightarrow \mathcal A_2$ between C$^*$-algebras is a $*$-isomorphism. 

Moreover, if $\mathcal A_1 = B(\mathcal H_1)$ and $\mathcal A_2 = B(\mathcal H_2)$ then there exists a unitary $u\in B(H_1,H_2)$ such that
\[
\psi(a) = uau^*, \quad \forall a\in B(\mathcal H_1).
\]
\end{corollary}

Now we can turn to the approximate versions of Kadison's results.
An $\epsilon$-approximate Jordan homomorphism is a linear map $\psi: \mathcal A \rightarrow \mathcal B$ such that
\[
\|\psi(ab) + \psi(ba) - \psi(a)\psi(b) - \psi(b)\psi(a)\| \,\leqslant\, \epsilon\|a\|\|b\|
\]
for all $a,b\in \mathcal A$. It is called an $\epsilon$-approximate Jordan $*$-homomorphism if it additionally satisfies
\[
\|\psi(a^*) - \psi(a)^*\| \,\leqslant\, \epsilon \|a\|
\]
for all $a\in \mathcal A$. As usual, isomorphism will indicate bijectivity.

\begin{theorem}[Theorem 2.12, Ili\v{s}evi\'c and Turn\v{e}k \cite{IT}]
Let $\mathcal A_1,\mathcal A_2$ be C$^*$-algebras such $K(\mathcal H_i) \subseteq \mathcal A_i\subseteq B(\mathcal H_i), i=1,2$. If $\psi : \mathcal A_1 \rightarrow \mathcal A_2$ is an $\epsilon$-approximate Jordan $*$-isomorphism with $\epsilon \in [0,10^{-6}]$ and $\|\psi^{-1}\| < \frac{1}{4\epsilon}$ (with the bound $\infty$ if $\epsilon = 0$), then there exists a unitary $u\in B(\mathcal H_1,\mathcal H_2)$ and a bound
\[
c(\epsilon) = \left(\sqrt 5 + \sqrt[4]{10}\right)^2(3+2\epsilon)\left(\frac{\epsilon}{1-3\epsilon}\left(20 + 50\epsilon + 56\epsilon^2 + 24\epsilon^3\right)\right)^{\frac{1}{2}} + 13\epsilon + \frac{49}{2}\epsilon^2 + \frac{23}{2}\epsilon^3
\]
such that
\begin{align*}
\|\psi(a) - uau^*\| & \leqslant c(\epsilon)\|a\|, \ \forall a\in \mathcal A_1, \quad \textrm{or} \\
\|\psi(a) - ua^tu^*\| & \leqslant c(\epsilon)\|a\|, \ \forall a\in \mathcal A_1 \,.
\end{align*}
\end{theorem}

\begin{theorem}[Theorem 3.4, Christensen \cite{Christensen}]
Let $\psi : B(\mathcal H_1) \rightarrow B(\mathcal H_2)$ be a completely positive, bijective linear map. If $t \in [0, 1/84]$ with $\|\psi\|\leqslant 1$ and $\|\psi^{-1}\| \leqslant 1+t$, then there exists a unitary $u\in B(\mathcal H_1,\mathcal H_2)$ such that
\[
\| \psi(a) - uau^*\| \leqslant (8.5t^{1/2} + 7t)\|a\|, \ \forall a\in B(\mathcal H_1).
\]
\end{theorem}

Note that Christensen's original result says ``isomorphism'' but this automatically implies $*$-isomorphism, which are all inner in the situation above.
% His main statement also has a coefficient of $8.5$ for $t^{1/2}$, but a closer inspection of his proof shows that this is merely an upper bound for $4\cdot(8.2\, \cdot\, 6/11)^{1/2} \approx 8.45953$.
This theorem also holds for a more general class of von Neumann algebras but we have reduced to the most significant case for our paper.

\begin{corollary}
Let $\Phi : B(\mathcal H) \rightarrow B(\mathcal H)$ be a unital self-adjoint bounded section of a unital completely positive map. If $\epsilon \in [0,1/84]$ and $\|\Phi\|_{\cbsdiam} \leqslant 1 + \epsilon$, then there exists $u\in B(\mathcal H)$ and a map $U(x) = uxu^*$ such that
\[
\|\Phi - U\|_{\cb} \leqslant 8.5\:\!\epsilon^{1/2} + 7\epsilon  \;.
\]
% Furthermore, if $\epsilon \in [0,1/168]$ and $\|\Phi\|_{\sdiam} \leqslant 1 + \epsilon$, then there exists $u\in B(\mathcal H)$ and a map $U(x) = uxu^*$ such that such that
% \[
% \bigl\lVert \Phi |_{B(\mathcal H)_{sa}} - U \bigr\rVert \leqslant 12\epsilon^{1/2} + 14\epsilon \;.
% \]
\end{corollary}
% \begin{corollary}
% Let $\Phi : B(\mathcal H) \rightarrow B(\mathcal H)$ be a unital self-adjoint bounded section of a unital completely positive map. If $\epsilon \in [0,1/84]$ and $\|\Phi\otimes \id_4\|_{\sdiam} \leqslant 1 + \epsilon$, then there exists a $u\in B(\mathcal H)$ such that
% \[
% \|\Phi(a) - uau^*\| \leqslant (8.5\epsilon^{1/2} + 7\epsilon)\|a\|, \ \forall a\in B(\mathcal H).
% \]
% \end{corollary}
\begin{proof}
This is an immediate consequence of Christensen's Theorem, Theorem \ref{thm:cbSpectralDiam}, and Proposition \ref{prop:optorNorm-sdiam-ineqs}.
\end{proof}

The reader will notice that the previous corollary has many unneeded words, $\Phi$ automatically is unital, self-adjoint, and bounded by virtue of being an inverse of a ucp map. However, the main goal is to formulate the type of general result that we would hope for. 

We conclude this section with a complementary bound, using the self-adjoint numerical diameter of the inverse of a bijective UCP map $\Psi$, to establish a bound for $\Psi$ away from all isometries:

\begin{proposition}
    Let $\Psi : B(\mathcal H) \rightarrow B(\mathcal H)$ be a UCP bijection, and let $\Phi = \Psi^{-1}$.
    Suppose that $\lVert \Phi \rVert_{\sdiam} \ge 1 + \epsilon$ for $\epsilon \ge 0$.
    Then for all maps $U(x) = uxu^*$ for $u : \mathcal H \to \mathcal H$ unitary, we have $\lVert \Psi - U \rVert_{\cb} \ge 2\epsilon (1+\epsilon)^{-1}$.
\end{proposition}
\begin{proof}

Let $\Psi: B(\mathcal H) \to B(\mathcal H)$ be a unital completely positive bijection, with inverse $\Phi$.
Let $u: \mathcal H \to \mathcal H$ be  unitary, $U(x) = uxu^*$\,, and $\delta = \lVert \Phi - U \rVert_{\cb}$.
Define $\Psi' = U^{-1} \circ \Psi$ and $\Phi' = \Phi \circ U$\,: then $\Psi' \circ \Phi' = \id_{B(\mathcal H)}$, and $\delta = \lVert \id_{B(\mathcal H)} -\:\! \Psi' \rVert_{\cb}$\,.

Let $\Delta = \id_{B(\mathcal H)} -\:\! \Psi'$, so that $\lVert \Delta \rVert_{\cb} = \delta$.
We also define the maps $\id_{sa}$, $\Psi_{sa}$, $\Phi_{sa}$, and $\Delta_{sa}$, representing the restrictions (respectively) of $\id_{B(\mathcal H)}$, $\Psi'$, $\Phi'$, and $\Delta$, to self-adjoint operators as arguments.

As $\Delta$ is a difference of unital maps, we have $\Delta(1_{\mathcal H}) = 0$.
By the Wittstock-Paulsen decomposition \cite{Haagerup}, for any $\gamma>0$, we may 
decompose $\Delta = \Delta_+ - \Delta_-$ as the difference of two completely positive maps, for which $\Delta_+(1_{\mathcal H}) = \Delta_-(1_{\mathcal H})$ and
\[
\|\Delta\|_{\cb} = \|\Delta_+ + \Delta_-\| - \gamma\, .
\]
%Note that
%\[
%    \bigl\lVert \Delta \bigr\rVert_{\cb}
%\;=\;
%    \inf_{a_j, b_j}
%    \;\;
%    \biggl\lVert
%        \sum_{j} a_j a_j^*
%    \biggr\rVert^{1/2}
%    \biggl\lVert
%        \sum_{j} b_j b_j^*
%    \biggr\rVert^{1/2}
%=\;
%    \inf_{a_j}
%    \;\;
%    \biggl\lVert
%        \sum_{j} a_j a_j^*
%    \biggr\rVert
%\;=\;
%    \bigl\lVert \Delta_+(1_{\mathcal H}) + \Delta_-(1_{\mathcal H}) \bigr\rVert ,
%\]
%where the infima range over generalised Choi--Kraus decompositions $\Delta(X) = \sum_j a_j \;\!X\;\! b_j^*$\,, and the second-to-last equality holds for $\Delta$ self-adjoint (for which we may take $b_j = \pm a_j$).
As $\Delta_+(1_{\mathcal H}) = \Delta_-(1_{\mathcal H})$, it then follows that
\[
    \delta+\gamma
\;=\,
    \bigl\lVert \Delta \bigr\rVert_{\cb} + \gamma
\,=\,
    \bigl\lVert 2\;\!\Delta_\pm(1_{\mathcal H}) \bigr\rVert
\,=\,
    2 \bigl\lVert \Delta_\pm \bigr\rVert.
\]
We also have
\[
    -\Delta_-(E) \;\le\; \Delta(E) \;\le\; \Delta_+(E)
\]
for all self-adjoint operators $E$, from which it follows that if $\lVert E \rVert = 1$ as well,
\[
    \bigl\lVert \Delta(E) \big\rVert 
    \;\le\;
    \max\;\Bigl\{ \bigl\lVert \Delta_+(E) \bigr\rVert,\, \bigl\lVert \Delta_-(E) \bigr\rVert \Bigr\}
    \;\le\;
    \bigl\lVert \Delta_\pm \bigr\rVert
    \;=\;
    \tfrac{1}{2}(\delta+\gamma);
\]
thus $\bigl\lVert \Delta_{sa} \bigr\rVert \le \tfrac{1}{2}(\delta+\gamma)$. Since this is true for every $\gamma > 0$ we have $\bigl\lVert \Delta_{sa} \bigr\rVert \le \tfrac{1}{2}\delta$.
% \smallskip
% \begin{itemize}
% \item
Note that
\[
    \id_{sa}
    \,=\,
    \Psi_{sa} \circ \Phi_{sa}
    \,=\,
    (\id_{sa} -\:\! \Delta_{sa}) \circ \Phi_{sa}
    \,=\,
    \Phi_{sa} - (\Delta_{sa} \circ \Phi_{sa}).
\]
Applying the triangle inequality, we then obtain
\[
\qquad
\qquad
    1
    \;\ge\;
    \Bigl\lvert 
        \,\bigl\lVert \Phi_{sa} \bigr\rVert 
        \,-\, 
        \bigl\lVert \Delta_{sa} \circ \Phi_{sa} \bigr\rVert
    \,\Bigr\rvert
    \;\ge\;
    \bigl(1 - \bigl\lVert \Delta_{sa} \bigr\rVert\bigr) \:\! \bigl\lVert \Phi_{sa} \bigr\rVert
    \;\ge\;
    (1 - \tfrac{1}{2}\delta) \:\! \bigl\lVert \Phi_{sa} \bigr\rVert
    \;.
\]
Suppose that $\lVert \Phi \rVert_{\sdiam} \ge 1 + \epsilon$ for $\epsilon \ge 0$.
By Proposition \ref{prop:optorNorm-sdiam-ineqs}, we then have
\[
    1
    \;\ge\;
    (1-\tfrac{1}{2}\delta) \,\lVert \Phi \rVert_{\sdiam}
    \;\ge\;
    (1-\tfrac{1}{2}\delta) (1+\epsilon),
\]
from which it follows that $\delta \ge 2\epsilon (1+\epsilon)^{-1}$.
\end{proof}

%\begin{corollary} Probably not true
%Let $\Phi: \mathfrak A \to \mathfrak B$ be a conservative section of a unital completely positive map between unital C$^*$-algebra. Then $\|\Phi\|_{\sdiam}=1$ if and only if $\Phi$ is a completely isometric map.
%\end{corollary}
%\begin{proof}
%The only thing to prove is isometric implies completely isometric. To this end, %assume $\Phi$ is isometric. 
%\end{proof}

\section{Translating finite-dimensional self-adjoint maps by the trace}
We now turn to some finite-dimensional results about complete positivity and translations by the trace.

\begin{proposition}[Choi \cite{Choi}]
For $n\in \mathbb N$ the linear map $\Psi_n: M_n(\mathbb C) \rightarrow M_n(\mathbb C)$ defined by
\[
\Psi_n(A) = n\:\mathrm{tr}(A)I - A
\]
is completely positive.
\end{proposition}
%\begin{proof}
%Choi (need citation, Paulsen Exercise 3.6) proved that the map
%\[
%M_n(\mathbb C) \ni A \ \mapsto \ n \operatorname{Diag}(A) - A
%\]
%is completely positive. The trace is a completely positive map and so  for %every $[A_{ij}] \in M_k(M_n(\mathbb C))$
%\begin{align*}
%n\:\mathrm{tr}^{(k)}([A_{ij}]) \ & = \ n\:\mathrm{tr}^{(k)}([\operatorname{Diag}(A_{ij})])
%\\ & \geqslant \ n[\operatorname{Diag}(A_{ij})])
%%\\ & =  \ n\operatorname{Diag}^{(n)}([A_{ij}])
%\\ & \geqslant \ [A_{ij}].
%\end{align*}
%%Therefore, $\Psi_n$ is completely positive.
%\end{proof}

%Of course, these constants are not optimal. For instance, when $n=1$ the constant 1 would suffice since
%\[
%\mathrm{tr}([a])I_1 - [a] = a - a = 0,
%\]
%but this is a trivial situation.

\begin{theorem}\label{thm:tracecp}
If $\Phi:M_n(\mathbb C) \rightarrow M_m(\mathbb C)$ is a self-adjoint linear map then there exists a scalar $\beta>0$ such that
\[
A\  \mapsto\  \Phi(A) + \beta\:\mathrm{tr}(A)I_m
\]
is a completely positive map.
\end{theorem}
\begin{proof}
By the Choi-Kraus decomposition we know that $\Phi = \Phi_+ - \Phi_-$, that is, $\Phi$ is the difference of two completely positive linear maps. Define 
\[
\varphi(A) = m\:\mathrm{tr}(\Phi_-(A)),
\]
which is positive as it is the composition of two positive maps.
Hence,
\begin{align*}
A \ &\mapsto \ \Phi(A) + \varphi(A)I_m
\\ & = \ \Phi_+(A) + m\:\mathrm{tr}(\Phi_-(A)) - \Phi_-(A)
\\ & = \ \Phi_+(A) + \Psi_m(\Phi_-(A))
\end{align*}
is a completely positive linear map.
Now, for $\|A\|=1$ such that $A\geqslant 0$, we have $\mathrm{tr}(A) \geqslant 1$ and 
\[
\|\varphi\| \;=\; \varphi(I_n) \;=\; m\:\mathrm{tr}(\Phi_-(I_n))
\;\leqslant\; 
m\:\mathrm{tr}(\|\Phi_-\|I_m)
\;=\;
m^2\|\Phi_-\|
\]
Thus, 
\[
A \;\mapsto\; m^2\:\|\Phi_-\|\cdot\mathrm{tr}(A) \,-\,   \varphi(A)
\]
is a positive linear functional, and so is completely positive.
Therefore, defining $\beta = m^2\|\Phi_-\|$ we get that
\begin{align*}
A &\ \mapsto\ \Phi(A) + \beta\:\mathrm{tr}(A)
\\ &\ = \ \Phi(A) + \varphi(A)I_m + \beta\:\mathrm{tr}(A) - \varphi(A)I_m
\end{align*}
is a completely positive map.
\end{proof}

Suppose $\Phi:M_n(\mathbb C) \rightarrow M_m(\mathbb C)$ is a self-adjoint linear map. Let $\beta>0$ be such that $\Phi + \beta I_m \cdot \mathrm{tr}$ be completely positive. Combining this with Theorem \ref{thm:cbSpectralDiam} we get
\[
\|\Phi\|_{\sdiam} = \|\Phi + \beta I_m \cdot \mathrm{tr} \|_{\sdiam} \leqslant \|\Phi + \beta I_m \cdot \mathrm{tr}\| = \Phi(I_n) + \beta n\,.
\]
One can work this out to an explicit constant which depends on $n$ by following through the previous proofs. It will be far from optimal but a bound is nice to have.

\begin{remark}
The previous proposition and theorem do not have infinite-dimensional equivalents. While one could replace the trace with a faithful state (norm 1, positive linear functional that is injective on the positive operators) the difficulty arises from the fact that this state is not bounded below on the positive operators. That problem aside we also have the additional problem that the constant in the proposition goes to infinity as $n$ increases (not that we claim that this constant is optimal).
\end{remark}

\begin{definition}
A linear map $\Phi: M_n(\mathbb C) \rightarrow M_m(\mathbb C)$ is called \textbf{scaled trace-preserving} if there exists $c\in \mathbb C$ such that $\mathrm{tr}(\Phi(A)) = c\:\mathrm{tr}(A)$ for every $A\in M_n(\mathbb C)$.
\end{definition}

\begin{lemma}
Let $\Phi: M_n(\mathbb C) \rightarrow M_m(\mathbb C)$ be a linear map. Then $\Phi$ is scaled trace-preserving if and only if $\Phi$ takes trace-zero matrices to trace-zero.
\end{lemma}
\begin{proof}
The forward direction is trivial. Conversely, let $c=\frac{1}{n}\mathrm{tr}\Big(\Phi(I_n)\Big)$.
Therefore, for any $A\in M_n(\mathbb C)$ we have that
\begin{align*}
    \mathrm{tr}(\Phi(A)) & = \mathrm{tr}\left(\Phi\Big(A - \frac{1}{n}\mathrm{tr}(A)I_n + \frac{1}{n}\mathrm{tr}(A)I_n\Big)\right)
\\ & = \mathrm{tr}\left(\Phi\Big(\frac{1}{n}\mathrm{tr}(A)I_n\Big)\right)
\; =\; \frac{1}{n}\mathrm{tr}\big(\Phi(I_n)\big)\:\mathrm{tr}(A)
\; =\; c\:\mathrm{tr}(A).
\qedhere
\end{align*}
\end{proof}

\begin{corollary}
If $\Phi: M_n(\mathbb C) \rightarrow M_m(\mathbb C)$ is a linear scaled trace-preserving map then so is $\Phi + \alpha\:\mathrm{tr}\cdot I_m$
\end{corollary}

\begin{theorem}
Let $\Phi: M_n(\mathbb C) \to M_m(\mathbb C)$ be a self-adjoint paraunital map, for which $\mathrm{null}(\Phi) \subseteq I_m \cdot \mathbb C$.
  If $\Phi$ is scaled trace-preserving, then there is a scalar $\gamma \in \mathbb R$ for which $$A \ \mapsto \ \Phi(A) + \gamma\:\mathrm{tr}(A)I_m$$ is the section of a completely  positive map. 
\end{theorem}
\begin{proof}
Let $c\in \mathbb R$ give $\Phi(I_n) = cI_m$. If $c<0$ then replace $\Phi$ with $\Phi + (|c|+1)\mathrm{tr}\cdot I_m$ which is still self-adjoint, paraunital and scaled trace-preserving. Thus we may assume that $\Phi$ is injective with $c>0$. Let $k\in\mathbb R$ be such that $\mathrm{tr}(\Phi(A)) = k\:\mathrm{tr}(A)$ for all $A\in M_n(\mathbb C)$.

Now $\mathcal S = \Phi(M_n(\mathbb C))$ is an operator system in $M_m(\mathbb C)$ and so we can define $\Psi: \mathcal S \rightarrow M_n(\mathbb C)$ by $\Psi = \Phi^{-1}$. Hence, $\Psi$ is paraunital, $\Psi(I_m) = \frac{1}{c}I_n$, and self-adjoint. We can extend $\Psi$ to a self-adjoint map on all of $M_m(\mathbb C)$. In particular, we can choose a basis for $M_m(\mathbb C)$ made of self-adjoint matrices such that the first $k$ form a basis for $\mathcal S$. Extend $\Psi$ by sending the remaining basis elements to 0. Thus, $\Phi$ is a section of a paraunital self-adjoint map.

By Theorem \ref{thm:tracecp} there exists $\beta\in \mathbb R$ such that $\Psi + \beta\:\mathrm{tr}\cdot I_n$ is completely positive. For $\gamma\in \mathbb R$ and any $A\in M_n(\mathbb C)$ we have that 
\begin{align*}
(\Psi + \beta\:\mathrm{tr}&\cdot I_n)\circ(\Phi + \gamma\:\mathrm{tr}\cdot I_m)(A)
\\[1ex] & = \Psi\circ\Phi(A) + \Psi(\gamma\:\mathrm{tr}(A)I_m) + \beta\:\mathrm{tr}(\Phi(A))I_n + \beta\:\mathrm{tr}(\gamma\:\mathrm{tr}(A)I_m)I_n
\\ & = A + \gamma\:\mathrm{tr}(A)\frac{1}{c}I_n + \beta k \:\mathrm{tr}(A)I_n + \beta\gamma m \: \mathrm{tr}(A)I_n
\\ & = A + \left(\gamma\left(\frac{1}{c} + \beta m\right) + \beta k \right)\mathrm{tr}(A)I_n
\end{align*}
which used both paraunital and scaled trace-preserving.
Note that $\frac{1}{c} + \beta m > 0$. Therefore, for $\gamma = -\beta k\left(\frac{1}{c} + \beta m\right)^{-1}\in \mathbb R$ we have that $\Phi + \gamma\:\mathrm{tr}\cdot I_m$ is a section of $\Psi + \beta\mathrm{tr}\cdot I_n$, a completely positive map.
\end{proof}

Creating sections of positive maps by translating by multiples of the trace map turns out to be impossible in general if one does not assume scaled trace-preserving as the following example shows.

\begin{example}
Let $\Phi: M_2(\mathbb C) \rightarrow M_2(\mathbb C)$ be defined as
\[
\Phi\left(\left[\begin{matrix} a & b \\ c& d\end{matrix}\right]\right) = \left[\begin{matrix} a-d & b \\ c& a+d\end{matrix}\right]
\]
which is unital, self-adjoint, and injective, but not scaled trace-preserving.
Now for any $\gamma\in \mathbb R$ we have that
\[
(\Phi + \gamma\:\mathrm{tr}\cdot I_2)\left(\left[\begin{matrix} 1 & 0 \\ 0& -1\end{matrix}\right]\right)
= \left[\begin{matrix} 2 & 0 \\ 0& 0\end{matrix}\right],
\]
that is, it takes a non-positive to a positive.
Therefore, for every choice of $\gamma\in \mathbb R$, any section of $\Phi + \gamma\:\mathrm{tr}\cdot I_2$ is not positive.
\end{example}

It is important to note that such translations by scalar multiples of the trace do not change the numerical diameter of these maps. However, it does usually change the completely bounded numerical diameter.

\begin{example}
Consider the identity map $\id : M_2(\mathbb C) \rightarrow M_2(\mathbb C)$. It is clear that \\ $\|\id + \mathrm{tr}(\cdot)I_2\|_{\diam} = \|\id\|_{\diam} =1$. But consider 
\[
\left\| \left[\begin{matrix} 0_2 & I_2 \\ I_2 & 0_2\end{matrix}\right] \right\|_{\diam} = 2
\]
while
\[
\left\|\id + \mathrm{tr}(\cdot)I_2\left( \left[\begin{matrix} 0_2 & I_2 \\ I_2 & 0_2\end{matrix}\right]\right) \right\|_{\diam} = \left\| \left[\begin{matrix} 0_2 & 3I_2 \\ 3I_2 & 0_2\end{matrix}\right] \right\|_{\diam} = 6\,.
\]
Therefore, $\|\id + \mathrm{tr}(\cdot)I_2\|_{\cbdiam} > \|\id\|_{\diam}$.
\end{example}

We end this paper with an example of a section of a completely positive map that has many nice properties and exhibits a completely bounded numerical diameter that is strictly larger than the numerical diameter.

\begin{example}
For $n\geqslant 2$, consider the map $\Psi : M_n(\mathbb C) \rightarrow M_n(\mathbb C)$ given by 
\[
\Psi(A) = \frac{1}{n^2}A^T + \frac{n^2 - 1}{n^3}\mathrm{tr}(A)I_n\, .
\]
We can then calculate that 
\[
[\Psi(E_{ij})]_{i,j=1}^n = \frac{1}{n^2}[E_{ji}]_{i,j=1}^n + \frac{n^2-1}{n^2} I_{n^2} \geqslant 0
\]
because the non-diagonal part of $\frac{1}{n^2}[E_{ji}]_{i,j=1}^n$ has norm at most $\frac{n^2-n}{n^2}$. Hence, by Choi's Theorem \cite[Theorem 3.44]{Paulsenbook} $\Psi$ is completely positive. Notice as well that $\Psi$ is unital and trace-preserving.

It is straightforward to see that the inverse, and so the section, of $\Psi$ is $\Phi : M_n(\mathbb C) \rightarrow M_n(\mathbb C)$ defined by
\[
\Phi(A) = n^2 A^T - \frac{n^2 - 1}{n} \mathrm{tr}(A) I_n\, ,
\]
which is a unital self-adjoint map. By Example \ref{ex:growingcbdiam} and the fact that the numerical diameter ignores scalar multiples of the identity we see that $\|\Phi\|_{\diam} = n^2$. Therefore,  
\[
\|\Phi\|_{\cbdiam} \geqslant \|\Phi\|_{\cb} \geqslant n^3 - n^2 + 1 > n^2 = \|\Phi\|_{\diam}
\]
for $n\geqslant 2$.
\end{example}

\section*{Acknowledgements}

This work began whilst N.dB.\ was affiliated with the University of Oxford, where he was supported by the EPSRC National Hub in Networked Quantum Information Technologies (NQIT.org). C.R.\ was supported by the NSERC Discovery grant 2019-05430. The authors would like to thank MathOverflow for their introduction.


\begin{thebibliography}{99}

\bibitem{Arveson}
W. Arveson, \textit{Subalgebras of C*-algebras II}, Acta Math {\bf 128} (1972), 271-308.

\bibitem{BonsallDuncan}
F. Bonsall and J. Duncan, \textit{Numerical Ranges of Operators on Normed Spaces and of Elements of Normed Algebras}, London: Cambridge University Press, 1971.

\bibitem{BourinLee}
J-C. Bourin and E-Y. Lee, \textit{Numerical range and positive block matrices}, Bull.  Aust. Math. Soc. {\bf 103} (2021), 69-77. 

\bibitem{BourinLee2}
J-C. Bourin and E-Y. Lee, \textit{Eigenvalue inequalities for positive block matrices with the inradius of the numerical range}, International J. Math. {\bf 33} (2022).

\bibitem{BourinMhanna}
J.-C. Bourin and A. Mhanna, \textit{Positive block matrices and numerical ranges}, C. R. Acad.
Sci. Paris, {\bf 355} (2017), 1077–1081.

\bibitem{Cao}
N. Cao, \textit{Numerical Ranges and Applications in Quantum Information}, PhD Thesis, University of Guelph, 2021.


\bibitem{Christensen}
E. Christensen, {\em Perturbations of operator algebras},
Invent. Math. {\bf 43} (1977), 1-13.


\bibitem{CNM}
Mao-Ting Chien, Hiroshi Nakazato, Jie Meng,
{\em The diameter and width of numerical ranges},
Linear Algebra and its Applications,
{\bf 582}
(2019), 76--98.

\bibitem{Choi}
M. Choi, \textit{Positive Linear Maps on C*-Algebras}, Canad. J. Math. {\bf 24} (1972), 520-529


\bibitem{Farenick}
D. R. Farenick, \textit{Matricial extensions of the numerical range: a brief survey}, Linear and Multilinear Algebra {\bf 34} (1993), 197-211.


\bibitem{FlemingJamison} 
R. J. Fleming and J. E. Jamison, {\it Isometries on Banach spaces: function spaces}, 
Chapman \& Hall/CRC Monographs and Surveys in Pure and Applied Mathematics, 129, Chapman \& Hall/CRC, Boca Raton, FL, 2003. MR1957004



\bibitem{Grabiner}
S. Grabiner, {\em The spectral diameter in Banach algebras}, Proc. Amer. Math Soc. {\bf 91} (1984), 59--63.


\bibitem{Haagerup}
U. Haagerup, \textit{Injectivity and decomposition of completely bounded maps}, 
Operator algebras and their connections with topology and ergodic theory (Bu\c{s}teni, 1983), 170–222, Lecture Notes in Math. 1132, Springer, Berlin, 1985.

\bibitem{Hausdorff}
F. Hausdorff, \textit{Der wertvorrat einer bilinearform}, Mathematische Zeitschrift {\bf 3} (1919), 314–316.

\bibitem{IT}
D. Ili\v{s}evi\'c and A. Turn\v{e}k, {\em A quantitative version of Herstein's theorem for Jordan $*$-isomorphisms}, Linear and Multilinear Algebra, {\bf 64} (2016), 156--168.

\bibitem{Jung}
H. Jung, \textit{Über den kleinsten Kreis, der eine ebene Figur einschließt}, J. für die reine und angewandte Mathematik {\bf 137} (1910), 310-313.



\bibitem{Kadison}
R. Kadison, {\em Isometries of operator algebras}, Ann. Of Math. {\bf 54} (1951), 325--338.

\bibitem{Kribsetal}
D. Kribs, A. Pasieka, M. Laforest, C. Ryan, and M. Silva, \textit{Research problems on numerical ranges in quantum computing}, Linear and Multilinear Algebra {\bf 57} (2009), 491–502.

%\bibitem{McKenney}
%Paul McKenney (\url{https://mathoverflow.net/users/11233/paul-mckenney}), Almost isometric linear maps, URL (version: 2018-06-25): \url{https://mathoverflow.net/q/303626}

\bibitem{Parlett}
B. Parlett, {\em The spectral diameter as a function of the diagonal entries}, Numer. Lin. Alg. App. {\bf 10} (2003), 595--602.

\bibitem{Paulsenbook}
V. Paulsen, {\it Completely bounded maps and operator algebras},
Cambridge Studies in Advanced Mathematics {\bf 78}, Cambridge University Press, 2002.

\bibitem{Paulsen}
V. Paulsen, private correspondence, 2019.

%\bibitem{Ramsey-6Nov2017} C.~Ramsey.
%\newblock	``Retractions for completely positive unital maps, with particularly nice norms''.
%\newblock MathOverflow [\url{mathoverflow.net/q/285209}], 2017--11--06.

%\bibitem{Ramsey-16Dec2017} C.~Ramsey.
%\newblock	``Retractions for completely positive unital maps, and their effect on spectral diameter''.
%\newblock MathOverflow [\url{mathoverflow.net/q/288579}], 2017--12--16.

%\bibitem{Schur} 
%I. Schur, {\em Einige Bermerkungen zur determinanten theorie}, S.B. Preuss Akad. Wiss. Berlin {\bf 25} (1925), 454--463.

%\bibitem{Skoufranis}
%P. Skoufranis, {\em Numerical Ranges of Operators}, \url{https://pskoufra.info.yorku.ca/files/2016/07/Numerical-Range.pdf}



\bibitem{Stone}
M. H. Stone, \textit{Linear transformations in Hilbert space and their applications to analysis}, vol. 15. American Mathematical Society, 1932.


\bibitem{Toeplitz}
O. Toeplitz, \textit{Das algebraische analogon zu einem satze von fejér}, Mathematische Zeitschrift {\bf 2} (1918), 187–197.


\bibitem{Tomiyama}
J. Tomiyama, \textit{On the Transpose Map of Matrix Algebras}, Proc. Amer. Math. Soc. {\bf 88} (1983), 635-638.

\bibitem{Watrous-2005} J.~Watrous.
\newblock \emph{Notes on Super-operator Norms Induced by Schatten Norms}.
\newblock Quantum Info.\ Comput.~\textbf{5} (58--68), 2005.
%\newblock [arXiv:quant-ph/0411077]

\end{thebibliography}
\end{document}